\documentclass[a4paper]{amsart}
\usepackage[T1]{fontenc}
\usepackage{lmodern}
\usepackage{amssymb}
\usepackage[all]{xy}
\usepackage[inline,shortlabels]{enumitem}
\usepackage{nicefrac,stmaryrd}
\usepackage{mathtools}
\usepackage[british]{babel}
\usepackage{microtype}

\usepackage{tikz-cd, tikz}
\usepackage[pdftitle={Topological correspondences-I},
pdfauthor={Rohit Dilip Holkar},
  pdfsubject={}
]{hyperref}
\usepackage[lite]{amsrefs}
\usepackage{amsthm, amsfonts}
\usepackage{nicefrac, color}
\usepackage{caption}

\numberwithin{equation}{section}
\theoremstyle{plain}
\newtheorem{theorem}[equation]{Theorem}
\newtheorem{lemma}[equation]{Lemma}
\newtheorem{proposition}[equation]{Proposition}

\newtheorem{corollary}[equation]{Corollary}

\theoremstyle{definition}
\newtheorem{definition}[equation]{Definition}

\theoremstyle{remark}
\newtheorem{remark}[equation]{Remark}

\newtheorem{example}[equation]{Example}

\newcommand*{\defeq}{\mathrel{\vcentcolon=}}

\newcommand{\R}{\mathbb R}
\newcommand{\C}{\mathbb C}

\newcommand*{\base}[1][H]{{#1}^{(0)}}%Base space of a groupoid
\newcommand{\nb}{\nobreakdash} %nonbreakable hyphen; needs amsmath
\newcommand{\7}{\backslash}
\newcommand{\inverse}{^{-1}}
\newcommand{\supp}{\textup{supp}}
\newcommand{\HH}{\mathrm H}%(co)homology
\newcommand{\CC}{\mathrm C}%(co)chain complex
\newcommand{\homeo}{\approx}
\newcommand{\iso}{\simeq}

 \newcommand*{\dd}{\textup d}%differential, used in dx in integrals

\newcommand{\inpro}[2]{\left\langle#1 \mathbin, #2 \right\rangle}

\newcommand*{\Star}{*\nb-}
\newcommand*{\Bound}{\mathbb B}%adjointable operators on a Hilbert module
\newcommand*{\Comp}{\mathbb K}%compact operators on a Hilbert module
\newcommand*{\Cst}{\textup C^*}%C*-algebra
\newcommand*{\Cred}{\textup C^*_\textup r}%reduced group C*-algebra
\newcommand*{\Hils}[1][H]{\mathcal{#1}}%Hilbert space
\newcommand*{\Mult}{\mathcal M}%multiplier algebra
\newcommand{\Ltwo}{\mathcal L^2}%L2 functions

\title[Topological correspondences]{Topological construction of $C^*$\nobreakdash-correspondences for groupoid $C^*$\nobreakdash-algebras}
\author{Rohit Dilip Holkar}
\email{rohit.d.holkar@gmail.com}
\address{Department of Mathematics, Federal University of Santa Catarina, 88.\,040-900, Florianop\'olis, SC, Brazil}

\begin{document}

\maketitle{}

\begin{abstract}
Let $(G,\alpha)$ and $(H,\beta)$ be locally compact groupoids with Haar systems. We define a topological correspondence from $(G,\alpha)$ to $(H,\beta)$ to be a $G$-$H$\nb-bispace $X$ on which $H$ acts properly, and $X$ carries a continuous family of measures which is $H$\nb-invariant and each measure in the family is $G$\nb-quasi invariant. We show that  a topological correspondence produces a $\Cst$\nb-correspondence from $\Cst(G,\alpha)$ to $\Cst(H,\beta)$. We give many examples of topological correspondences. 
\end{abstract}

\tableofcontents{}

\section*{Introduction}
        
A \emph{$\Cst$\nb-algebraic correspondence} $H$ from a 
$\textup{C}^*$\nb-algebra $A$ to~$B$ is an 
$A$-$B$-bimodule which is a Hilbert $B$\nb-module and $A$ acts on
$H$ via the adjointable operators in a non-degenerate fashion. Let $A=\Cst(G,\alpha)$ and $B=\Cst(H,\beta)$ where the ordered pairs $(G,\alpha)$ and $(H,\beta)$ consist of a locally compact groupoid and a Haar system for it. Given a $G$-$H$\nb-bispace~$X$ carrying an $H$\nb-invariant family of measures such that each measure in the family is $G$\nb-quasi-invariant, we show that if the $H$\nb-action is proper, then $C_c(X)$ can be completed into a $\Cst$\nb-correspondence from $\Cst(G,\alpha)$ to $\Cst(H,\beta)$. This work is an extension of some part of my thesis~\cite{mythesis}, where we worked with Hausdorff typologies with certain countability assumptions. In the present work we get rid of the Hausdorffness and the countability hypotheses.

Morita equivalence of $\Cst$\nb-algebras is defined by the
existence of an imprimitivity bimodule, a special kind of
$\Cst$\nb-correspondence. In the well-known
result (\cite{Muhly-Renault-Williams1987Gpd-equivalence}) that a Morita equivalence between two locally compact
groupoids with Haar systems induces a Morita equivalence between the
groupoid $\Cst$\nb-algebras the imprimitivity module is constructed directly from a bispace giving the Morita equivalence of the two groupoids. The Hausdorff case of Morita equivalence between locally compact groupoids is discussed in~\cite{Muhly-Renault-Williams1987Gpd-equivalence}, and a much general and non-Hausdorff situation is studied in~\cite{Renault1985Representations-of-crossed-product-of-gpd-Cst-Alg}*{Definition 5.3}. Which extra structure or conditions are needed for a bispace to give only a $\Cst$\nb-algebraic correspondence instead of a Morita equivalence?

In general, we need a family of measures on the bispace as an extra structure to get started. In the Morita equivalence case, a family of measures on the bispace appears automatically (See Example~\ref{exa:ME-is-top-corr}). The the family of measures must be invariant for the right action and each measure of the family must be quasi-invariant for the left action.  We also need that the right action is proper.

We use the \Star{}category of a locally compact groupoids introduced in~\cites{Renault1985Representations-of-crossed-product-of-gpd-Cst-Alg, Holkar-Renault2013Hypergpd-Cst-Alg} to prove that certain actions and a bilinear form are well-defined (Equations~\ref{def:left-right-action} and~\ref{def:inner-product}). The process of constructing a $\Cst$\nb-correspondence from a topological correspondence is divided into two main parts: the first part is to construct the
Hilbert module and the second one is to define the representation of the left groupoid $\Cst$\nb-algebra on this Hilbert module. For the first part, we use
the representation theory of groupoids and the transverse measure
theory introduced by Renault
in~\cite{Renault1985Representations-of-crossed-product-of-gpd-Cst-Alg}. 
In the second part, our motivation and techniques are derived from the theory of quasi-invariant measures for locally compact groups (\cite{Folland1995Harmonic-analysis-book}*{Section 2.6}). 

As a $\Cst$\nb-correspondence from a $\Cst$\nb-algebra $A$ to $B$ induce a representation of $B$ to that of $A$, a topological correspondence from a locally compact groupoid with a Haar system $(G,\alpha)$ to $(H,\beta)$ induce a representation of $H$ to that of $G$. Renault proves this in~\cite{Renault2014Induced-rep-and-Hpgpd}.
\smallskip

A locally compact, Hausdorff space is a locally compact groupoid with a Haar system, and so is a locally compact group. A well-known fact about groupoid equivalence is that two spaces are equivalent if and only if they are homeomorphic and two groups are equivalent if and only if they are isomorphic. But since any continuous map between spaces gives a topological correspondence and so does a group homomorphism, a topological correspondence is far more general than an equivalence.

We give many examples of topological correspondences, most of which are the analogues of the standard examples of $\Cst$\nb-correspondences.

In Example~\ref{exa:cont-function-as-corr}, we show that a continuous map $f\colon X\to Y$ between spaces gives a topological correspondence from $Y$ to $X$. Example~\ref{exa:gp-homo-as-corr} shows that a continuous group homomorphism $\phi\colon G\to H$ gives a topological correspondences from $G$ to $H$.  Theorem~\ref{thm:mcorr-gives-ccorr} gives that the topological correspondences, the one from $Y$ to $X$ and  the one from $G$ to $H$, produce $\Cst$\nb-correspondences from $C_0(Y)$ to $C_0(X)$ and $\Cst(G)$ to $\Cst(H)$, respectively. It is easy to see that the $\Cst$\nb-correspondence from $C_0(Y)$ to $C_0(X)$ is exactly the one give by the \Star{}homomorphism $f^*\colon  C_0(Y) \to \Mult(C_0(X))$ as in the theory of commutative $\Cst$\nb-algebras. However, it is not equally easy to see that the $\Cst$\nb-correspondence given by the group homomorphism $\phi$ agrees with the one which is induced by the \Star{}homomorphism $\phi_*\colon \Cst(G)\to \Cst(H)$. 

 Example~\ref{exa:proper-gp-homo-as-corr} shows that if the group homomorphism in Example~\ref{exa:gp-homo-as-corr} is a proper map, then we get a topological correspondence from $G$ to $H$.

Let $E^0$ and $E^1$ be locally compact, Hausdorff and second countable spaces,
and let $s,r\colon E^{1}\to E^0$ be continuous maps. Let 
$\lambda=\{\lambda_e\}_{e\in E^0}$ be a continuous family of measures along
$s$. By applying the definition of a topological correspondence it is
 straightforward to check that $s,r$ and $\lambda$ give a
topological correspondence from $E^0$ to itself.  Muhly and
Tomforde (\cite{Muhly-Tomforde-2005-Topological-quivers}*{Definition 3.1}) call this correspondence a \emph{topological quiver}. They construct a $\Cst$\nb-correspondence associated to a topological quiver in
\cite{Muhly-Tomforde-2005-Topological-quivers}*{Section 3.1} and the
construction in~\cite{Muhly-Tomforde-2005-Topological-quivers} is exactly the construction of a $\Cst$\nb-correspondence from a topological correspondence. Muhly and Tomforde define the $\Cst$\nb-algebra associated to a topological quiver (\cite{Muhly-Tomforde-2005-Topological-quivers}*{Definition 3.17}) which includes a vast class of $\Cst$\nb-algebras: graph $\Cst$\nb-algebras, $\Cst$\nb-algebras of topological graphs, $\Cst$\nb-algebras of branched coverings, $\Cst$\nb-algebras associated with topological relations are all associated to a topological quiver \cite{Muhly-Tomforde-2005-Topological-quivers}*{Section 3.3}.

Topological quivers justify our use of families of measures in the definition of a topological correspondence. At a first glance, the families of measures and their quasi-invariance for the left action might look artificial. However, as discussed on page~\pageref{para:use-of-adjo-funct}, the quasi-invariance of families of measures is natural to ask for.

We show that the notion of correspondences introduced by Macho Stadler and O'uchi (\cite{Stadler-Ouchi1999Gpd-correspondences}), the generalised morphisms introduced by Buneci and Stachura (\cite{Buneci-Stachura2005Morphisms-of-lc-gpd}) are topological correspondences. See Examples~\ref{exm:Stadler-Ouchi-correspondence} and~\ref{exa:buneci-stachura}, respectively.

 In~\cite{Tu2004NonHausdorff-gpd-proper-actions-and-K}, Tu defines locally proper generalised homomorphism for locally compact groupoids with Hausdorff space of units. Example~\ref{exm:Stadler-Ouchi-correspondence}, which shows that a correspondence in the sense of Macho Stadler and O'uchi is a topological, also shows that a locally proper generalised homomorphism is a topological correspondence. However, Tu proves that a locally proper generalised homomorphism induces a $\Cst$\nb-correspondence between the reduced $\Cst$\nb-algebras of the groupoids (\cite{Tu2004NonHausdorff-gpd-proper-actions-and-K}*{Proposition 2.28}). Whereas our result involves the full $\Cst$\nb-algebras of the groupoids. We know that the result of Tu holds for topological correspondences when the bispace involved in the topological correspondence is Hausdorff and second countable, and the right action is amenable~\cite{mythesis}*{Proposition 2.3.4}.
\smallskip

Following is the sectionwise description of the contents.

\paragraph{Section 1 (Preliminaries):}
In this section we mention our conventions, notation, and rewrite some standard definitions and results.

A notion of cohomology for Borel groupoids is introduced in~\cite{Westman1969Gpd-cohomology} by Westman. In~\cite{Renault1980Gpd-Cst-Alg}*{Chapter 1, page 14}, Renault discusses a continuous version of the same cohomology. We need a groupoid equivariant continuous version of this cohomology. For this purpose we define action of a groupoid on another groupoid and then define the equivariant cohomology for Borel and continuous groupoids.
\smallskip

\paragraph{Section 2 (Topological correspondences):}
This section contains the main construction. Immediately after the definition of topological correspondence (Definition~\ref{def:correspondence}) we discuss the role of the adjoining function.

 Let $(X,\lambda)$ be a topological correspondence from $(G,\alpha)$ to $(H,\beta)$. Then we write the formulae of the actions of $C_c(G)$ and $C_c(H)$ on $C_c(X)$. These actions make $C_c(X)$ into a $C_c(G)$-$C_c(H)$\nb-bimodule. We also define the formula of a $C_c(H)$\nb-valued bilinear map on $C_c(X)$ which, we latter prove, is a positive bilinear map.
We complete this setup to get a $\Cst$\nb-correspondence. The process, as mentioned earlier, is divided into two parts:
constructing a $\Cst(H,\beta)$\nb-Hilbert module $\Hils(X)$ and defining a
representation of $\Cst(G,\alpha)$ on this Hilbert module. 

We advise the reader to jump to Section~\ref{cha:examples-theory} after the discussion that follows Definition~\ref{def:correspondence} to have a look at some examples.
\smallskip

\paragraph{Section 4 Examples:}

This section contains examples of topological
correspondences.

\section{Preliminaries}
\label{cha:loc-cpt-hausdorff-gpd}
\subsection{Groupoids}
\label{sec:gpd-prelim}

 The reader should be familiar with the theory of locally compact groupoids (\cites{Anantharaman-Renault2000Amenable-gpd,Renault1980Gpd-Cst-Alg,Renault1985Representations-of-crossed-product-of-gpd-Cst-Alg,PatersonA1999Gpd-InverseSemigps-Operator-Alg} and~\cite{Tu2004NonHausdorff-gpd-proper-actions-and-K}). 

A groupoid $G$ is a small category in which every arrow is invertible. Except a few instances, we denote the space of arrows of groupoid $G$ by the letter $G$ itself, rather than the more precise symbol $G^{(1)}$. We denote the space of units of $G$ by $\base[G]$. The source and range maps are denoted by $s_G$ and $r_G$, respectively. The inverse of an element $\gamma\in G$ is denoted by $\gamma\inverse$. There are instances when we need to write $\gamma\mapsto\gamma\inverse$ as a function $G\to G$ and then we denote the function by $\textup{inv}_G$.

A pair $(\gamma,\gamma')\in G\times G$ is called composable if $s_G(\gamma)=r_G(\gamma')$. Sometimes we abuse the language by saying `$\gamma,\gamma'\in G$ are composable' by which we mean that the pair $(\gamma,\gamma')$ is composable. For $A,B\subseteq \base[G]$ define
\begin{align*}
G^A&=\{\gamma\in G: r_G(\gamma)\in A \} = r_G\inverse(A),\\
G_A&=\{\gamma\in G: s_G(\gamma)\in A\} = s_G\inverse(A) \textnormal{ and }\\
G^A_B&=G^A\cap G_B=\{\gamma\in G: r_G(\gamma)\in A \textup{ and }
  s_G(\gamma)\in B\}.
\end{align*}
When $A=\{u\}$ and $B=\{v\}$ are singletons, we write $G^u$, $G_v$ and
$G^u_v$ instead of $G^{\{u\}}$, $G_{\{v\}}$ and $G^{\{u\}}_{\{v\}}$, 
respectively. For $u\in\base[G]$, $G^u_u$ is a group. It is 
called the \emph{isotropy} group at $u$.

A topological groupoid and measurable groupoid have their standard meanings (\cite{Anantharaman-Renault2000Amenable-gpd}).

We call a subset $A\subseteq X$ of a topological space $X$ \emph{quasi-compact} if every open cover of $A$ has a finite subcover. And $A$ is called \emph{compact} if it is quasi-compact and Hausdorff. The space $X$ is called locally compact if every point $x\in X$ has a compact neighbourhood. All the spaces considered in this article are locally compact. For a locally compact space $X$ by $C_c(X)_0$ we denote the set of functions $f$ on $X$ such that $f\in C_c(V)$ where $V\subseteq X$ is open Hausdorff, and $f$ is extended outside $V$ by $0$ (See \cite{Tu2004NonHausdorff-gpd-proper-actions-and-K}*{Section 4}).\label{page:Cc} By $C_c(X)$ we denote the linear span of functions in $C_c(X)_0$. The functions in $C_c(X)$ need not be continuous on $X$ but they are Borel. Furthermore, if $\mu$ is a positive $\sigma$\nb-finite Radon measure on $X$, then $C_c(X)\subseteq \Ltwo(X,\mu)$ is dense.

As in~\cite{Renault1985Representations-of-crossed-product-of-gpd-Cst-Alg}, we call a topological groupoid $G$ locally compact if $G$ is a locally compact topological space and $\base[G]\subseteq G$ is Hausdorff. In this case $r_G\inverse(u)\subseteq G$ (and equivalently, $s_G\inverse(u)\subseteq G$) are Hausdorff for all $u\in\base[G]$.

If $X$ and $Y$ are topological spaces, $X\homeo Y$ means $X$ and $Y$ are homeomorphic. If $G$ and $H$ are group(oids)s, then $G\iso H$ means $G$ and $H$ are isomorphic via a group(oid) homomorphism. In group case, this means $G$ and $H$ are isomorphic. All the measures we work with are positive, Radon and $\sigma$\nb-finite.

 For groupoid actions we do not assume that the momentum maps are open
 or surjective. However, it is well-known that for a locally compact groupoid with a Haar system the source map (equivalently the range map) is automatically open. We need that each measure in a family of measures along a continuous open map $f\colon  X\to Y$ is non-zero, but it need not have full support.

\subsection{Proper actions and families of measures}
\label{sec:prop-act-and-measures}

Since we shall not come across any case where there are more than one different left (or right) action of a groupoid~$G$ on a space $X$, we denote the momentum map by $r_X$ (respectively, $s_X$). When we write `$X$ is a left (or right) $G$\nb-space'  without specifying the momentum map, the above convention will be tacitly assumed and then in such instances the momentum map is $r_X$ (respectively, $s_X$).

 Let $A\subseteq \base[G]$. For a left $G$\nb-space $X$ and a right $G$\nb-space $Y$ we define $X^A, X^u, Y_A$ and $Y_u$ similar to $G^A, G^u, G_A$ and $G_u$.

Let $X, Y$ and $Z$ be spaces, and let $f\colon X\to Z$ and $g\colon Y\to Z$ be maps. We denote the fibre product of $X$ and $Y$ over $Z$ by $X\times_{f,Z,g}Y$. When there is no confusion about the maps $f$ and $g$, we simply write $X\times_{Z}Y$ instead of $X\times_{f,Z,g}Y$.

 Let $G$ be a groupoid and $X$ a left $G$\nb-space. By $G\ltimes X$ we denote the transformation groupoid. Its space of arrows is $G\times_{s_G,\base[G],r_X}X$, which we prefer denoting by $G\times_{\base[G]}X$. Recall that a Haar system on $G$ induces a Haar system on $G\ltimes X$.

 Let $G$ be a locally compact groupoid with open range map and $X$ a locally compact $G$\nb-space. Then the quotient map $X\to X/G$ is open. If the action of $G$ is proper, then $X/G$ is locally compact. If $X$ is Hausdorff (or Hausdorff and second countable) the $X/G$ is also Hausdorff (Hausdorff and second countable, respectively).

Let $X$ be a left $G$\nb-space. For subsets $K\subseteq G$ and $A\subseteq X$ with $s_G(K)\cap r_X(A) \neq \emptyset$, define $K\!A = \{\gamma x: \gamma\in K, x\in A \textup{ and } (\gamma, x)\in G\times_{\base[G]}X\}$. If $s_G(K)\cap r_X(A) = \emptyset$, then we define $AK=\emptyset$. By an abuse of notation, for $x\in X$ we write $Kx$ instead of $K\{x\}$. The meaning of $\gamma A$ for $\gamma\in G$ is similar. For a right action, we define $A\!K$, $xK$ and $A\gamma$ similarly.

Let $G$ and $H$ be groupoids and $X$ be a space on which $G$ acts from the left and $H$ from the right. If for all $\gamma\in G$ , $x\in X$ and $\eta\in H$ with $s_G(\gamma)=r_X(x)$ and $s_X(x)=r_H(\eta)$ we have $s_X(\gamma x)=s_X(x)$, $r_X(x\eta)=r_X(x)$, and \[(\gamma x)\eta=\gamma(x\eta), \] then we call $X$ a $G$-$H$\nb-bispace.

\begin{definition}[Invariant continuous family of measures]\label{def:gpd-invariant-measure-family}
  Let $H$ be a locally compact groupoid, and let $X$ and $Y$ be locally compact right $H$\nb-spaces. Let $\pi\colon  X\to Y$ be an $H$\nb-equivariant continuous map. An
  \emph{$H$\nb-invariant continuous family of measures} along $\pi$ is a family of
  Radon measures $\lambda=\{\lambda_y\}_{y\in Y}$ such
  that:
  \begin{enumerate}[label=\roman*)]
  \item each $\lambda_y$ is defined on $\pi\inverse(y)$;
  \item (invariance) for all composable pairs $(y,\eta) \in
    Y\times_{\base}H$, the condition $\lambda_{y}\eta=\lambda_{y\eta}$ holds;
  \item (continuity condition) for $f\in C_c(X)$ the function 
    $\Lambda(f)(y) \defeq \int_{\pi\inverse(y)}
    f\,\dd\lambda_{y}$ on $Y$ is continuous.
  \end{enumerate}
\end{definition}
We clarify that in the above definition the measure $\lambda_{y}\eta$ is given by
$\int f\,\dd \lambda_{y}\eta=\int f(x\eta)\,\dd\lambda_y(x)$ for $f\in C_c(X)$.

If for each $y\in Y$, $\supp(\lambda_y) =\pi\inverse(y)$, we say the family of measures $\lambda$ has \emph{full support}. If there is
a continuous function $f$ on $X$ with  $\Lambda(f)=1$ on
$\pi(X)$, we say that $\lambda$ is \emph{proper}. Lemma 1.1.2 in
\cite{Anantharaman-Renault2000Amenable-gpd} says that, in the continuous
case, $\lambda$ is proper if and only if $\lambda_y\neq 0$ for all
$y\in Y$. Thus if $\lambda$ is continuous and has full support, then
$\lambda$ is proper. In the rest of the article we assume that every family of measures we work with is proper.

Let $\textup{Pt}$ be the trivial point group(oid). If $X$ and $Y$ are spaces and $\pi\colon  X\to Y$ is a continuous map,
then $\pi$ is a $\textup{Pt}$\nb-equivariant map between
$\textup{Pt}$\nb-spaces. A continuous $\textup{Pt}$\nb-invariant
family of measures along $\pi$ is simply called a \emph{continuous family of measures along} $\pi$.

We denote families of measures by small Greek letters. For a
given family of measures, the corresponding integration function that appears
in the continuity condition in
Definition~\ref{def:gpd-invariant-measure-family} will be denoted by
the Greek upper case letter used to denote the family of measures. For
$\alpha$, $\beta$ and $\mu$ it will be $A$, $B$ and $M$,
respectively. 
\begin{definition}
\label{def:haar-system-and-more}
\begin{enumerate}
  \item Let $H$ be a groupoid, $X$ a left $H$\nb-space. An
    $H$\nb-invariant continuous family of measures along the momentum
    map $r_X$ is called a \emph{left $H$\nb-invariant continuous
      family of measures on $X$}. A \emph{right $H$\nb-invariant
      continuous family of measures on $X$} is defined analogously.
  \item For a groupoid $H$, a \textit{Haar system} on $H$ is a left
    $H$\nb-invariant continuous family of
    measures with full support on $H$ for the \textit{left multiplication} action of $H$ on itself.
  \end{enumerate}
\end{definition}

\begin{lemma}
  \label{cor:lemma:r-systems-makes-r-open-1}
Let $H$ be a groupoid, let $\pi\colon  X\to Y$ be a continuous
  $H$\nb-map between the $H$\nb-spaces $X$ and $Y$ and let $\lambda$
  be a continuous family of measures along $\pi$. If $\lambda_y$ has
  full support for all $y\in \pi(X)$, then $\pi$ is an open
  map onto its image.
\end{lemma}
\begin{proof}
  Consider the map $\pi\colon  X\to \pi(X)$ and then the proof is similar to the one of Proposition 2.2.1 in \cite{PatersonA1999Gpd-InverseSemigps-Operator-Alg}.
\end{proof}

\begin{corollary}
\label{cor:range-source-maps-for-haar-gpd-open}
  If $(H,\beta)$ is a groupoid with a Haar system, then the range and
  source maps are open.
\end{corollary}
\begin{proof}
  Lemma~\ref{cor:lemma:r-systems-makes-r-open-1} implies that the
  range map $r_H$ is open. Since $s_H=r_H\circ \textup{inv}_H$ and $\textup{inv}_H$ is a
  homeomorphism, $s_H$ is open.
\end{proof}
 We need the following two lemmas from~\cite{Renault1985Representations-of-crossed-product-of-gpd-Cst-Alg}.
\begin{lemma}[Lemme 1.1,
  \cite{Renault1985Representations-of-crossed-product-of-gpd-Cst-Alg}]
\label{lemma:french-1.1}
  Let $X$ and $Y$ be spaces, let $\pi\colon  X\to Y$ be an open surjection
  and let $\lambda$ be a family of measures with full support along $\pi$. For every open
  $U\subseteq X$ and for a non-negative function $g \in C_c(\pi(U))$,
  there is a non-negative function $f\in C_c(U)$ with $\Lambda(f) =g$. 
\end{lemma}
\begin{lemma}[Lemme 1.2,
  \cite{Renault1985Representations-of-crossed-product-of-gpd-Cst-Alg}]
\label{lemma:french-1.2}
  Let $X$, $Y$ and $Z$ be spaces, let $\pi$ and $\tau$ be open surjections
  from $X$ and  $Y$ to $Z$, respectively. Let $\pi_2$ denote the
  projection  from the fibre product $X\times_ZY$ onto the second factor 
  $Y$. Assume that for each $z\in Z$, there is a measure $\lambda_z$
  on $\pi\inverse(z)$. For each $y\in Y$ define the measure
  ${\lambda_2}_y=\lambda_{\tau(y)}\times \delta_y$, where
  $\delta_y$ is the point-mass at $y$. Then $\lambda$ is
  continuous if and only if\; $\lambda_2$ is continuous.
\end{lemma}

Let $H$ be a locally compact groupoid, and let $X$ and $Y$ be locally compact right $H$\nb-spaces. For $x\in X$ the equivalence class of $x$ in $X/H$ is denoted by $[x]$. The map $\pi$ induces a map from $X/H$ to $Y/H$, which we denote by $[\pi]$. Let $\pi\colon X\to Y$ be an $H$\nb-equivariant map and $\lambda$ an $H$\nb-invariant continuous family of measures along $\pi$. Then $\lambda$ induces a continuous family of measures $[\lambda]$ along $[\pi]\colon X/H\to Y/H$. The proof is similar to the proof of Proposition 2.15 in~\cite{mythesis}. For $f\in C_c(X/H)$ and $[y]\in Y/H$ the measure $[\lambda]^{[y]}$ is defined as
\begin{equation}
  \label{eq:quotient-family-of-measures}
\int_{X/H} f\,\dd[\lambda]^{[y]} \defeq \int_{X} f([x])\,\dd\lambda^y(x).  
\end{equation}
\begin{example}
  \label{ex:def-of-beta-x}
  Let $X$ be a proper right $H$\nb-space. Let $\beta$ be a Haar system on $H$. Then Lemma~\ref{lemma:french-1.2} says that $\beta_1=\{\delta_x\times\beta_{s_X(x)}\}_{x\in X}$ is a continuous family of measures along the projection map $\pi_1\colon X\times_{\base}H\to X$. Take the quotient by the action of $H$ to get a continuous family of measure  $[\pi_1]$ along the map $[\pi_1]\colon (X\times_{\base}H)/H\to X/H$. Identifying $(X\times_{\base}H)/H= X$ gives that $[\pi_1]$ is the quotient map, and a small computation gives that for $f\in C_c(X)$ and for $[x]\in X/H$,\[\int_X f\,\dd[\beta_1]^{[x]}=\int_{H^{s_X(x)}} f(x\eta)\,\dd\beta^{s_X(x)}(\eta).\]
 In the rest of the article we write $\beta_X$ instead of $[\beta_1]$.
\end{example}

\subsection{Cohomology for groupoids}\label{sec_Cohomology_for_groupoids}

 In this section, the groupoids and maps are
assumed to be Borel. The whole
discussion goes through when the Borel properties are replaced by the 
continuous properties. 
\begin{definition}[Action of a groupoid on another groupoid]\label{def:action-of-gpd-on-gpd}
A left action of a groupoid $G$ on another groupoid $H$ is given by
maps $r_{H,G}\colon  H\to\base[G]$ and $a\colon  G\times_{s_G,\base[G],r_{H,G}}H\to H$ which satisfy the following conditions:

    \begin{enumerate}[label=\roman*)]
    	\item if $\eta,\eta'\in H$ are composable, $\gamma\in G$
          with $s_G(\gamma)=r_{H,G}(\eta)=r_{H,G}(\eta')$, then
          $a(\gamma,\eta),a(\gamma,\eta')\in H$ are composable and 
\[a(\gamma,\eta)a(\gamma,\eta')=a(\gamma, \eta\eta');\]
     \item if $u\in\base[G]$, then $a(u, \eta)=\eta$ for all
          $\eta\in r_{H,G}\inverse(u)\subseteq H$;
	\item if $\gamma,\gamma'\in G$ are composable, then $(\gamma,
          a(\gamma',\eta))\in G\times_{s_G,\base[G],r_{H,G}}H$ and
          \[a(\gamma\gamma',\eta) = a(\gamma, a(\gamma',\eta)).\]  
    \end{enumerate}
\end{definition}
To simplify the notation, we write $\gamma\cdot\eta$ or simply
$\gamma\eta$ for $a(\gamma,\eta)$, and $G\times_{\base[G]}H$ for $G\times_{s_G,\base[G],r_{H,G}}H$. Then (i) and (ii) above read
$\gamma\cdot(\eta\eta')=(\gamma\cdot\eta)(\gamma\cdot\eta')$ and
$(\gamma\gamma')\cdot\eta=\gamma\cdot(\gamma'\cdot\eta)$,
respectively. We call the map $r_{H,G}$ the momentum map for the
action and $a$ the action map. When the momentum and action
maps are continuous (or Borel) the action is called continuous 
(or Borel, respectively).

It is not hard to see that Definition~\ref{def:action-of-gpd-on-gpd} gives an action of $G$ on $H$ via invertible functors. When $G$ is a group, our definition matches Definition 1.7 in~\cite{Renault1980Gpd-Cst-Alg}*{Chapter 1}, which is the action of a group on a groupoid by invertible functors. A proof of this fact is below.

\begin{lemma}
\label{lemma:gpd-to-gpd-action-implies-gp-to-gp-action}
When $G$ is a group, an action of $G$ on $H$ as in
Definition~\ref{def:action-of-gpd-on-gpd} above is the same as the
action in~\cite{Renault1980Gpd-Cst-Alg}*{Definition 1.7, Chapter 1},  that is, there is a homomorphism $\phi\colon  G\to \textup{Aut}(H)$ which gives the action. Here $\textup{Aut}(H)$ is the set of all invertible functors from
$H$ to itself.
\end{lemma}
\begin{proof}
Definition~\ref{def:action-of-gpd-on-gpd} implies~\cite{Renault1980Gpd-Cst-Alg}*{Definition 1.7, Chapter 1}: 
For $\gamma\in G$ define $ \phi(\gamma)(\eta)=\gamma\cdot\eta$.
We first prove that each $\phi(\gamma)$ is a functor from $H$ to
itself.

Note that an element $u$ in a groupoid is a unit if and only if $u$
is composable with itself and $u^2=u$. If
$u\in \base[G]$, then $\phi(\gamma)(u)=\phi(\gamma)(uu)=
\phi(\gamma)(u) \phi(\gamma)(u)=(\phi(\gamma)(u))^2$. Hence for each unit $u\in \base$,
$\phi(\gamma)(u)\in H$ is a unit. Condition (i) of Definition~\ref{def:action-of-gpd-on-gpd} gives that for each $\gamma\in G$, 
$\phi(\gamma)(\eta\eta')=\phi(\gamma)(\eta)\phi(\gamma)(\eta')$. This
proves that $\phi(\gamma)$ is functor for each $\gamma\in G$.
\smallskip 

Now we show that each of the $\phi(\gamma)$ is invertible. Condition (iii) of
Definition~\ref{def:action-of-gpd-on-gpd} gives that $\gamma\mapsto \phi(\gamma)$ is a homomorphism. Use
(iii) of Definition~\ref{def:action-of-gpd-on-gpd} to see
that $\phi(\gamma)$ is invertible:
\[
\phi(\gamma)\phi(\gamma\inverse)(\eta)=\phi(\gamma\gamma\inverse)(\eta)=\phi(r_G(\gamma))(\eta)=\eta=\textup{Id}_H(\eta).
\]

Similarly, $\phi(\gamma\inverse)\phi(\gamma)=\textup{Id}_H$. Thus
$\phi(\gamma)\inverse = \phi(\gamma\inverse)$. Hence $\phi(\gamma)\in
\textup{Aut}(H)$.

~\cite{Renault1980Gpd-Cst-Alg}*{Definition 1.7, Chapter 1} implies Definition~\ref{def:action-of-gpd-on-gpd} implies: Proof of this part is routine checking of the conditions in Definition~\ref{def:action-of-gpd-on-gpd}.
\end{proof}

A continuous (and Borel) version of Lemma~\ref{lemma:gpd-to-gpd-action-implies-gp-to-gp-action} can be proved along same lines merely by adding continuity (or Borelness) of the action map and the
momentum map and the continuity (Borelness) of the group homomorphism $\phi$.

\begin{example}
\label{exa:gp-gpd-action}
Let $G$ be a groupoid and $H$ a space. Then an action of $G$ on
$H$ is the same as an action of $G$ on $H$ viewed as a groupoid. In this case,
condition (i) in Definition~\ref{def:action-of-gpd-on-gpd} is irrelevant and 
then the definition matches the usual one.
\end{example}

\begin{example}\label{exm:bispace-gives-gpd-action-on-gpd}
Let $G$ and $H$ be groupoids and let $X$ be a $G$-$H$\nb-bispace. Define an
action of $G$ on the transformation groupoid $X\rtimes H$ by $\gamma (x, \eta)
\defeq (\gamma x, \eta)$. The momentum map for this action
is $(x,\eta) \mapsto r_X(x) \in \base[G]$. Let $(x\eta, \eta'),
(x,\eta)\in X\times_{\base} H$ be composable elements then $(\gamma
x\eta, \eta')(\gamma x,\eta)$ are composable and \[\gamma \cdot(x,\eta) \cdot\gamma\,(x\eta, \eta')=(\gamma x,\eta) (\gamma x\eta, \eta')= (\gamma x, \eta\eta') = \gamma(x, \eta\eta').\]
This verifies (i) of Definition~\ref{def:action-of-gpd-on-gpd}. The
other conditions are easy to check. Thus
 $H$ acts on the groupoid $G\ltimes X$ in our sense. This is an important example for us.
\end{example}

Let $H$ be a Borel groupoid and assume that $H$ acts on a Borel groupoid
$G$. Let $\base[G]$ and $G^{(1)}$ have the usual meaning. For $n= 2, 3, \ldots$ define \[G^{(n)} = \{(\gamma_0, \dots,\gamma_{n-1}) \in \underbrace{G\times
G\times\dots\times G}_{\text{$n$-times}}: s_G(\gamma_i) = r_G(\gamma_{i+1})
\text{ for } 0\leq i< n-2\}.\]
\begin{definition}
\label{def:cochain-complex}
 Let $G$, $H$ be Borel groupoids, let $A$ be an abelian Borel group
 and let $H$ act on $G$. The $A$\nb-valued $H$\nb-invariant
 Borel cochain complex $(B\mathcal{C}_H^{\bullet}(G; A), d^{\bullet})$
 is defined as follows:
 	\begin{enumerate}[label=\textit{\roman*})]
		\item The abelian groups $B\CC_H^n$ are:
			\begin{enumerate}
		  	\item $B\CC_H^0(G; A) \defeq \{ f: \base[G]
                          \to A: f \;\textnormal{ is an }\;H\text{-invariant Borel map}\}$;
			\item for $n>0$ $ B\CC_H^n(G; A) \defeq$ $\{
                          f\colon  G^{(n)} \rightarrow A$ : $f$ is an~$H$\nb-invariant Borel map and
                          $f(\gamma_0,\dots,\gamma_{n-1})= 0$
                          if $\gamma_i \in \base[G]$ for some $0\leq i\leq n-1\}$,
			\end{enumerate}
		\item the coboundary map $d$ is:
			\begin{enumerate}
				\item $d^0\colon  B\CC_H^0(G; A) \rightarrow B\CC_H^1(G; A)$ is $d^0(f)(\gamma) = f(s_G(\gamma)) - f(r_G(\gamma))$,
				\item for $n>0$, $d^n\colon  B\CC_H^n(G; A)
                                  \rightarrow B\CC_H^{n+1}(G; A)$ is 
\begin{multline*}
d^n(f)((\gamma_0,\dots,\gamma_{n}))= f(\gamma_1,\dots,\gamma_{n})\\
+\sum\limits_{i=1}^n (-1)^i f(\gamma_0,\dots,
\gamma_{i-1}\gamma_i,\dots,\gamma_{n}) + (-1)^{n+1}
f(\gamma_0,\dots,\gamma_{n-1}).
\end{multline*} 
                       \end{enumerate}
	\end{enumerate}
\end{definition}

The $n$-th cohomology group of this complex is the $n$-th
$H$-invariant \emph{Borel cohomology} of~$G$ for $n\geq 0$, and it is
denoted by $\HH_{\mathrm{Bor}, H}^n(G; A)$. By adding the action of
$H$ to all the maps and spaces, one can make sense of the machinery and the results in~\cite{Westman1969Gpd-cohomology}*{\S 1 and \S 2}.
\begin{remark}
   \label{rem:char-of-low-dim-cocycle}
Any $H$\nb-invariant Borel function $f$ on $\base[G]$ is a $0$-cochain. A cochain $f\in B\CC_H^0(G; A)$ is a cocycle \emph{iff} $d^0(f) = 0$ which is true \emph{iff} $f$ is constant on the orbits of $\base[G]$.
A cochain $k \in B\CC_H^1(G; A)$ is a cocycle \emph{iff} $k(\gamma_0)-
k(\gamma_0\gamma_1)+k(\gamma_1) = 0$ for all composable $\gamma_0$ and
$\gamma_1$, which is equivalent to saying that $k$ is an $H$\nb-invariant Borel groupoid homomorphism.
\end{remark}
We drop the suffixes $B$ and $\mathrm{Bor}$ and write merely $\CC_H^0(G; A)$ and $\HH_H^n(G; A)$. 
\begin{remark}
  \label{rem:relation-between-good-zero-cocycles}
Let $b,b'\in \CC_H^0(G; A)$ with $d^0(b)=d^0(b')$. Then $c=b-b'$ is a 0\nb-cocycle since $d^0(c)=d^0(b)-d^0(b')=0$. Remark~\ref{rem:char-of-low-dim-cocycle} now gives that $c$ is constant on the orbits of $\base[G]$.
%  Put $c=b-b'$. Then for all $\gamma\in G$,
%  \begin{align*}
%  d^0(b)(\gamma)&=d^0(b')(\gamma)\\ 
% b(s_G(\gamma))-b(r_G(\gamma))&=  b'(s_G(\gamma))-b'(r_G(\gamma))\\
% b(s_G(\gamma))-b'(s_G(\gamma))&=b(r_G(\gamma))-b'(r_G(\gamma))\\
% c(s_G(\gamma))&=c(r_G(\gamma))
%  \end{align*}
Thus $c$ is a function on $\base[G]/G$.
\end{remark}

\subsection{Quasi-invariant measures}
\label{sec:quasi-invar-meas}

Let $(G,\alpha)$ be a locally compact groupoid with a Haar
system. Using $\alpha$ we get a right invariant family of
measures~$\alpha\inverse$ along the source map. For $f\in C_c(G)$ set $\int
f(\gamma)\,\dd\alpha\inverse_u(\gamma)=\int f(\gamma\inverse)\,\dd\alpha^u(\gamma)$. Let $X$ be a left $G$\nb-space and let $\mu$ be a measure on $X$. We define a measure $\mu\circ\alpha\inverse$ on the space $G\times_{\base[G]}X$ by
\[
\int_{G\times_{\base[G]}X}f\,\dd(\mu\circ\alpha\inverse)=\int_X\int_{G^{r_X(x)}} f(\gamma\inverse,x)\,\dd\alpha^{r_X(x)}(\gamma)\,\dd\mu(x)
\]
for $f\in C_c(G\times_{\base[G]}X)$. Similarly we define the measure $\mu\circ\alpha$.
\begin{definition}[Quasi-invariant measure]\label{def:quasi-invariant-measure-on-gpd}
  Let $(G,\alpha)$ be a groupoid with a Haar system and $X$ a $G$\nb-space. A measure $\mu$ on $X$ is
  called $(G,\alpha)$\nb-\emph{quasi-invariant} if $\mu\circ\alpha$ and
  $(\mu\circ\alpha)\circ\textup{inv}_{G\ltimes X}$ are equivalent.
\end{definition}

In the above definition, $\textup{inv}_{G\ltimes X}$ is the inverse
function on the transformation groupoid $G\ltimes X$. Thus for $f\in
C_c(G\times_{\base[G]}X)$
\[(\mu\circ\alpha)\circ\textup{inv}_{G\ltimes X}(f)=\int_X\int_{G^{r_X(x)}} f(\gamma,\gamma\inverse x)\,\dd\alpha^{r_X(x)}(\gamma)\dd\mu(x).\]
When the groupoid $(G,\alpha)$ with a Haar measure in the discussion is fixed and there is no possibility of confusion, we write `$\mu$ is a quasi-invariant measure'
\begin{remark}\label{rem:mod-fun-for-action-gpd-is-ae-homomorphism}
As in~\cite{Anantharaman-Renault2000Amenable-gpd} or
\cite{Buneci-Stachura2005Morphisms-of-lc-gpd}, it can be shown that the
Radon-Nikodym derivative
$\dd(\mu\circ\alpha)/\dd(\mu\circ\alpha\inverse)$ is a
$\mu\circ\alpha$\nb-almost everywhere a groupoid
homomorphism from the transformation groupoid $G\ltimes X$ to
$\R^*_+$.
\end{remark}
\begin{remark}
   A $(G,\alpha)$\nb-quasi-invariant
   measure $\mu$ is $G$\nb-invariant if and only if the Radon-Nikodym
   derivative $\dd(\mu\circ\alpha)/\dd(\mu\circ\alpha\inverse)=1$ 
  $\mu\circ\alpha$\nb-almost everywhere on $G\times_{\base[G]}X$.
\end{remark}
A \emph{measured groupoid} is a triple $(G,\alpha,\mu)$ where $(G,\alpha)$ is a locally compact groupoid with a Haar system and $\mu$ is a $(G,\alpha)$\nb-quasi-invariant measure on $\base[G]$.
\begin{definition}[Modular function]\label{def:modular-function}
 The \emph{modular function} of a measured groupoid $(G,\alpha,\mu)$ is the
  Radon-Nikodym derivative $\dd(\mu\circ\alpha)/\dd(\mu\circ\alpha\inverse)$.
\end{definition}
The reason to use the article \textit{the} for the modular function is
that if $\Delta$ and $\Delta'$ are two modular, then $\Delta
= \Delta'$ $\mu\circ\alpha$-almost everywhere.

\subsection{\texorpdfstring{$\Cst$}{C*}-correspondences}
\label{sec:some-more-definitios}

 We shall use the theory of Hilbert modules and assume that the reader is
familiar with the basics of the theory (For example,~\cite{Lance1995Hilbert-modules}).
\begin{definition}\label{def:Cst-corr}
  Let $A$ and $B$ be $\Cst$ algebras. A
  \emph{$\Cst$\nb-correspondence} from $A$ to $B$ is a Hilbert
  $B$\nb-module $\Hils$ with a non-degenerate *\nb-representation $A\to \Bound_{B}(\Hils)$.
\end{definition}
Here $\Bound_{B}(\Hils)$ denotes the $\Cst$\nb-algebra of adjointable operators on $\Hils$. A Hilbert $B$\nb-module $\Hils$ is \emph{full} if the linear span of the image of $\Hils\times\Hils$ under the inner product map is dense in $B$. We call a $\Cst$\nb-correspondence $\Hils$ from $A$ to $B$ a
\emph{proper} correspondence if $A$ acts on
$\Hils$ by compact operators, that is, the action of $A$ is given by a
non-degenerate *\nb-representation $A\to\Comp_B(\Hils)$.
\begin{definition}
  An \emph{imprimitivity bimodule} from $A$ to $B$ is an $A$-$B$\nb-bimodule $\Hils$ such that
  \begin{enumerate}[label=\roman*)]
  \item $\Hils$ is a full left Hilbert $A$\nb-module with an inner product $\prescript{}{*}{\inpro{}{}}$;
  \item $\Hils$ is a full right Hilbert $B$\nb-module with an inner product $\inpro{}{}_*$;
  \item $(\Hils, \prescript{}{*}{\inpro{}{}})$ is a correspondence from $B$ to $A$;
  \item $(\Hils, \inpro{}{}_*)$ is a correspondence from $A$ to $B$;
  \item for $a,b,c\in\Hils$ $a\inpro{b}{c}_* =\, \prescript{}{*}{\inpro{a}{b}}c$.
  \end{enumerate}
\end{definition}
 Our notion of $\Cst$\nb-correspondence (Definition~\ref{def:Cst-corr}) is wider, in the sense that many authors demand that the Hilbert module involved in a $\Cst$\nb-correspondence is \emph{full}, or for some authors a $\Cst$\nb-correspondence is what we call a \emph{proper} correspondence.

In~\cite{Rieffel1974Induced-rep}, Rieffel shows that an $A$-$B$\nb-imprimitivity bimodule induces an isomorphism between the representation
categories of $B$ and $A$. In general, if $\Hils$ is a $\Cst$\nb-correspondence from $A$ to $B$, then $\Hils$ induces a
functor from $\textup{Rep}(B)$, the representation category of $B$, to $\textup{Rep}(A)$.

\section{Topological correspondences}
\label{cha:meas-corr}

\begin{definition}[Topological correspondence]
  \label{def:correspondence}
A \emph{topological correspondence} from a locally compact groupoid $G$ with a Haar
system $\alpha$ to a locally compact groupoid $H$ equipped with a Haar system $\beta$ is a pair $(X, \lambda)$ where:
	\begin{enumerate}[label=\textit{\roman*})]
		\item $X$ is a locally compact $G$-$H$-bispace,
                \item the action of $H$ is proper,
              	\item $\lambda = \{\lambda_u\}_{u\in\base}$ is an
                  $H$\nb-invariant proper continuous family of measures along the momentum map $s_X\colon  X\to\base$,
		\item  $\Delta$ is a continuous function $\Delta: G
                  \times_{\base[G]} X \rightarrow \R^+$ such that for each $u
                  \in \base$ and $F\in C_c(G\times_{\base[G]}X)$,
                  \begin{multline*}
                    \int_{X_{u}} \int_{G^{r_X(x)}} F(\gamma\inverse,
                    x)\, \dd\alpha^{r_X(x)}(\gamma)\, \dd\lambda_{u}(x) \\=
                    \int_{X_{u}} \int_{G^{r_X(x)}} F(\gamma,
                    \gamma\inverse x)\, \Delta(\gamma, \gamma\inverse
                    x) \, \dd\alpha^{r_X(x)}(\gamma)
                    \,\dd\lambda_{u}(x). 
                \end{multline*}
	\end{enumerate}
\end{definition}

If $\Delta'$ is another function that satisfies condition (iv)
in Definition~\ref{def:correspondence}, then $\Delta = \Delta'$ 
$\lambda_u\circ\alpha$-almost everywhere for each $u\in\base$. As
both $\Delta$ and $\Delta'$ are continuous, we get $\Delta=\Delta'$. We
call the function $\Delta$ \emph{the adjoining function} of the
correspondence $(X,\lambda)$.
\begin{remark}
\label{rem:specialities-of-corr}
  Note that we do not need that the momentum maps $s_X$ and $r_X$ are
  open surjections. We also do not demand that the family of measures
  $\lambda$ has full support, hence the Hilbert
  module in the resulting $\Cst$\nb-correspondence need not be full.
  The resulting $\Cst$\nb-correspondence need not be proper, as well.
\end{remark}

\begin{remark}
\label{rem:quasi-inv-of-lambda}
Referring to Definition~\ref{def:quasi-invariant-measure-on-gpd},
we can see that Condition (iv) in Definition~\ref{def:correspondence} says that
the measure $\alpha\times\lambda_u$ on $G\times_{\base[G]}X_u$ is
$(G,\alpha)$\nb-quasi-invariant for each $u\in \base$ where the measure $\alpha\times\lambda_u$ is defined as follows: for $f\in C_c(G\times_{\base[G]}X_u)$,
\[\int_{G\times_{\base[G]}X_u} f\,\dd(\alpha\times\lambda_u)=\int_{X_{u}} \int_{G^{r_X(x)}} f(\gamma\inverse,x)\, \dd\alpha^{r_X(x)}(\gamma)\, \dd\lambda_{u}(x).\]
\end{remark}

In short, ``\emph{A topological correspondence from $G$ to $H$ is a pair $(X,\lambda)$ where $X$ is a $G$-$H$\nb-bispace and $\lambda$ is an $H$\nb-invariant family of measures  on $X$ indexed by $\base$ and each measure in $\lambda$ is $G$\nb-quasi-invariant.}''

\begin{remark}\label{rem:Delta-is-ae-homomorphism}
 As in~\cite{Anantharaman-Renault2000Amenable-gpd} or \cite{Buneci-Stachura2005Morphisms-of-lc-gpd}, it can be shown that $\Delta$ restricted to $G\times_{\base[G]}X_u$ is $\alpha\times \lambda_{u}$\nb-almost everywhere a groupoid homomorphism for all $u\in\base$. So the function $\Delta$ in (iv) of Definition~\ref{def:correspondence} is a continuous 1\nb-cocycle on the groupoid $G\ltimes X$. We shall use this fact in many computations.
\end{remark}

\begin{remark}\label{rem:Delta-is-right-invariant}
Example~\ref{exm:bispace-gives-gpd-action-on-gpd} gives a right
action of $H$ on $G\ltimes
X$. In~\cite{Buneci-Stachura2005Morphisms-of-lc-gpd}, Buneci and
Stachura use an adjoining function exactly like us. The topological correspondence Buneci and Stachura define is a special case of our construction
(Example~\ref{exa:buneci-stachura}). They show that the adjoining
function in their case is $H$\nb-invariant (see~\cite{Buneci-Stachura2005Morphisms-of-lc-gpd}*{Lemma 11}). In the similar fashion
we may prove that $\Delta$ is $H$\nb-invariant under the right action of $H$, that is,\[\Delta(\gamma, x\eta) =\Delta(\gamma, x) \] 
\noindent for all composable triples $(\gamma, x, \eta) \in G\times_{s_G,\base[G],r_X}X\times_{s_X,\base,r_H} H$. Thus, in
fact, $\Delta\colon  G\ltimes X/H\to\R^*_+$. Now
Remark~\ref{rem:Delta-is-ae-homomorphism} can be made finer by saying
that $\Delta$ is an $H$\nb-invariant continuous 1-cocycle on the groupoid $G\ltimes X$.

\end{remark}
\paragraph{Use of the family of measures and the adjoining function:}
\label{para:use-of-adjo-funct}
In the following discussion we explain the role of the adjoining function.
Let $(X,\lambda)$ be a topological correspondence from
$(G,\alpha)$ to $(H,\beta)$ with $\Delta$ as the adjoining
function.  We make $C_c(X)$ into a $C_c(H)$-module using the
same formula as in~\cite{Muhly-Renault-Williams1987Gpd-equivalence}
or~\cite{Stadler-Ouchi1999Gpd-correspondences}. 
To make $C_c(X)$ into a $\Cst(H,\beta)$\nb-pre-Hilbert module, we need to define a
$C_c(H)$\nb-valued inner product on $C_c(X)$. The formula for
this inner product cannot be copied directly from
either~\cite{Muhly-Renault-Williams1987Gpd-equivalence}
or~\cite{Stadler-Ouchi1999Gpd-correspondences}. This formula has to be modified,
and it uses the family of measures $\lambda$. 

Talking about the left action, for $\phi\in C_c(G)$ and $f\in C_c(X)$
\cite{Muhly-Renault-Williams1987Gpd-equivalence} and~\cite{Stadler-Ouchi1999Gpd-correspondences} define $\phi\cdot f \in C_c(X)$ by
\begin{equation}
  \label{eq:disc-classical-left-action}
(\phi\cdot f)(x) = \int_G \phi(\gamma) f(\gamma\inverse x) \, \dd \alpha^{r_X(x)}(\gamma).  
\end{equation}

For our definition of topological correspondence, the action of $C_c(G)$ on the $\Cst(H,\beta)$\nb-pre-Hilbert module $C_c(X)$ defined by Formula~\ref{eq:disc-classical-left-action} is not necessarily an action by adjointable operators. For $\phi$ and~$f$ as above we define the left action by
\begin{equation}
  \label{eq:disc-our-left-action}
(\phi\cdot f)(x) \defeq \int_G \phi(\gamma) f(\gamma\inverse x) \,
\Delta^{1/2}(\gamma, \gamma\inverse x) \; \dd
\alpha^{r_X(x)}(\gamma).  
\end{equation}
We shall see that the adjoining function gives a nice scaling factor for the action of $C_c(G)\subseteq \Cst(G,\alpha)$ and makes this action a \Star{}homomorphism. This is the reason we call $\Delta$ the \emph{adjoining} function. 

Two examples of topological correspondences are: an equivalence between groupoids (See~\cite{Muhly-Renault-Williams1987Gpd-equivalence} or Definition~\ref{def:gpd-equivalence}) and the correspondence of Marta Stadler and O'uchi (See~\cite{Stadler-Ouchi1999Gpd-correspondences} or Example~\ref{exm:Stadler-Ouchi-correspondence}). For equivalences and Macho Stadler-O'uchi correspondences the adjoining function is the constant function $1$, and then formulae~\eqref{eq:disc-classical-left-action} and~\eqref{eq:disc-our-left-action} match. To understand the role of $\Delta$ the reader may have a look at Lemma~\ref{lemma:left-action-self-adj}.

To support the necessity of the adjoining function, consider a toy example: Let $G$ be a locally compact group and let it act on a locally compact (possibly Hausdorff) space $X$ carrying a measure $\lambda$. The left  multiplication action of $C_c(G)$ on $C_c(X)\subseteq\Ltwo(X,\lambda)$ defined by Equation~\eqref{eq:disc-classical-left-action} is not necessarily bounded. To make this action bounded, it is sufficient that $\lambda$ is $G$\nb-quasi-invariant, which brings the adjoining function into the picture. Then the left action of $C_c(G)$ given by Equation~\eqref{eq:disc-our-left-action} becomes a \Star{}representation. This motivated us to introduce Condition (iv) in Definition~\ref{def:correspondence}. Buneci and Stachura~\cite{Buneci-Stachura2005Morphisms-of-lc-gpd} also use the adjoining function.

For the left multiplication action of $G$ on $G/K$, where $K$ is a closed subgroup of $G$,  the space $G/K$ always carries a   $G$-quasi-invariant measure $\lambda$. Hence there is a representation of $G$ on $\Ltwo(G/K, \lambda)$. Quasi-invariant measures and the corresponding \emph{adjoining functions} are studied very well in the group case, for example, see Section 2.6 of~\cite{Folland1995Harmonic-analysis-book}.

At this point, readers may peep into Section~\ref{cha:examples-theory} to see
some examples of adjoining functions.
\smallskip

We start with the main construction now. For $\phi \in C_c(G)$, $f \in C_c(X)$ and $\psi \in C_c(H)$ define functions $\phi\cdot f$ and $f\cdot \psi$ on $X$ as follows:

\begin{equation}\label{def:left-right-action}
 \left\{\begin{aligned}
(\phi\cdot f)(x) &\defeq \int_{G^{r_X(x)}} \phi(\gamma) f(\gamma\inverse x) \,
\Delta^{1/2}(\gamma, \gamma\inverse x) \; \dd \alpha^{r_X(x)}(\gamma),\\
(f\cdot\psi )(x) &\defeq \int_{H^{s_X(x)}} f(x\eta)\psi({\eta}\inverse) \; \dd
\beta^{s_X(x)}(\eta).
  \end{aligned}\right.
\end{equation}

For $f, g \in C_c(X)$ define the function $\langle f, g \rangle $ on $H$ by
\begin{align}\label{def:inner-product}
\langle f, g \rangle (\eta) &\defeq \int_{X_{r_H(\eta)}} \overline{f(x)} g(x\eta) \; \dd
\lambda_{r_H(\eta)}(x).
\end{align}
Most of the times we write $\phi f$ and $f\psi$ instead of $\phi\cdot f$
and $f\cdot \psi$. 
\begin{lemma}\label{lemma: cntty-of-lr-action}
Among the functions $\phi f$, $f\psi$ and $\inpro{f}{g}$ defined above, the first two are in $C_c(X)$ and the last one is in $C_c(H)$.
\end{lemma}
\begin{proof}
The proof follows from ~\cite{Renault1985Representations-of-crossed-product-of-gpd-Cst-Alg}*{Lemme 3.1} and the \Star{}category $C_c(H)$ used in~\cite{Renault1985Representations-of-crossed-product-of-gpd-Cst-Alg} or~\cite{Holkar-Renault2013Hypergpd-Cst-Alg}. Detail computations can be found in~\cite{mythesis}*{Section 3.3.1} in which Tables~{3.1} and~{3.2} list the equivalence of the operations in Definitions~\ref{def:left-right-action} and~\ref{def:inner-product} and those in the {\Star}category $C_c(H)$.

We prove that $\inpro{f}{g}\in C_c(H)$ and the remaining claims can be proved similarly. Let $\mathfrak{C}_c(H)$ denote the \Star{}category for $H$. Then $(H,\beta)$ and $(X,\lambda)$ are objects in $\mathfrak{C}_c(H)$. We identify $X$ with $(X\times_{\base}H)/H$ and then think of $f,g\in C_c((X\times_{\base}H)/H)$ as arrows from $(X,\lambda)$ to $(H,\beta)$. Then $\inpro{f}{g}=f^**_\lambda g$, where the latter is in $C_c((H\times_{r_H,\base,r_H}H)/H)$. Now identify $C_c((H\times_{r_H,\base,r_H}H)/H)$ with $C_c(H)$ to conclude the proof. 
\end{proof}
Both $C_c(G)$ and $C_c(H)$ are \Star{}algebras. Denote
the convolution product on them by $*$.
\begin{lemma}\label{lemma:bhel}
  Let $\phi , \phi^{\prime} \in C_c(G)$, $\psi , \psi^{\prime} \in C_c(H)$ and
$f, g, g' \in C_c(X)$. Then
  \begin{align}
  (\phi \ast \phi^{\prime})f& = \phi (\phi^{\prime} f), \label{eq:bhel-1}\\
  f(\psi \ast \psi^{\prime})& = (f \psi )\psi^{\prime},  \label{eq:bhel-2}\\
(\phi f)\psi &= \phi(f\psi), \label{eq:bhel-3}\\
   \langle f, g+g' \rangle & = \langle f, g \rangle + \langle f, g' \rangle,\label{eq:bhel-4}\\  
  \langle f\,, g \rangle ^{\ast}& = \langle g\,, f \rangle, \label{eq:bhel-5}\\
   \langle f\,, g \rangle* \psi& = \langle f\,, g\psi
                                \rangle, \label{eq:bhel-6}\\
\langle \phi f\,, g \rangle& =\langle f\,,\phi^*g \rangle. \label{eq:bhel-7}
  \end{align}
\end{lemma}
\begin{proof}
Checking most of the equalities above are straightforward computations
involving obvious change of variables, using the invariances of the
families of measures, Fubini's theorem and some properties of the
adjoining function. We write two detail computations and hints for
computing the others.
Let  $\gamma\in G$, $x\in X$ and $\eta\in
H$. 

\noindent Equation~\eqref{eq:bhel-1}: 
\begin{align*}
&  ((\phi*\phi')f)(x)\\
&= \int_{G^{r_X(x)}} (\phi*\phi')(\gamma)f(\gamma^{-1}
  x)\Delta(\gamma, \gamma^{-1} x)^{1/2}\,\dd\alpha^{r_X(x)}(\gamma)\\
&= \int_{G^{r_X(x)}} \int_{G^{r_G(\gamma)}}
  \phi(\zeta)\phi'(\zeta^{-1} \gamma)f(\gamma^{-1} x) \Delta(\gamma, \gamma^{-1} x)^{1/2}\,\dd\alpha^{r_G(\gamma)}(\zeta) \,\dd\alpha^{r_X(x)}(\gamma).
\end{align*}
\noindent First apply Fubini's theorem and then change the variable $\gamma\mapsto \zeta\gamma$ and
use the invariance of $\alpha$ to
see that the last term equals
\begin{align*}
\int_{G^{r_G(\gamma)}}\int_{G^{r_X(x)}}  \phi(\zeta)\phi'(\gamma)f(\gamma^{-1}\zeta^{-1} x) \Delta(\zeta\gamma, \gamma^{-1}\zeta^{-1} x)^{1/2}\,\dd\alpha^{r_X(x)}(\gamma)\,\dd\alpha^{r_G(\gamma)}(\zeta).
\end{align*}

We observe that $(\zeta\gamma, \gamma^{-1}\zeta^{-1} x)=(\zeta,\zeta^{-1} x) (\gamma,\gamma^{-1}\zeta^{-1} x)$ in the transformation groupoid
$G\ltimes X$. This relation,
Remark~\ref{rem:Delta-is-ae-homomorphism} and the associativity of
the left action together allow us to write the previous term as
\begin{align*}
   & \int_{G^{r_G(\gamma)}}\int_{G^{r_X(x)}} 
  \phi(\zeta)\phi'(\gamma)f(\gamma^{-1}\zeta^{-1} x)\\ 
&\qquad\qquad \Delta(\zeta,\zeta^{-1} x)^{1/2}\Delta (\gamma,\gamma^{-1}\zeta^{-1} x)^{1/2} \,\dd\alpha^{r_X(x)}(\gamma)\,\dd\alpha^{r_G(\gamma)}(\zeta)\\
&=\int_{G^{r_G(\gamma)}}\phi(\zeta)\left(\int_{G^{r_X(x)}}\phi'(\gamma)f(\gamma^{-1}\zeta^{-1} x) \Delta (\gamma,\gamma^{-1}\zeta^{-1} x)^{1/2}  \,\dd\alpha^{r_X(x)}(\gamma)\right)\\
&\qquad\qquad \Delta(\zeta,\zeta^{-1} x)^{1/2}\,\dd\alpha^{r_G(\gamma)}(\zeta)\\
&=\int_{G^{r_G(\gamma)}}\phi(\zeta) \,(\phi' f)(\zeta^{-1} x)\,\Delta(\zeta,\zeta^{-1} x)^{1/2}\,\dd\alpha^{r_G(\gamma)}(\zeta)\\
&=(\phi(\phi'f))(x).
\end{align*}
\noindent Equation~\eqref{eq:bhel-2}: This computation is similar to
the above computation for Equation~\eqref{eq:bhel-1} except 
that there is no adjoining function here.

\noindent Equation~\eqref{eq:bhel-3}: This uses Fubini's theorem and
the $H$\nb-invariance of $\Delta$.

\noindent Equation~\eqref{eq:bhel-4}:  This is a direct computation.

\noindent Equation~\eqref{eq:bhel-5}: This involves a change of
variable and the right
invariance of the family of measures $\lambda$.

\noindent Equation~\eqref{eq:bhel-6}: This uses a change of variable,
the left invariance of $\beta$ and Fubini's theorem.

\noindent Equation~\ref{eq:bhel-7}:
\begin{align}\label{eq:left-action-adjble}
\langle \phi f, g\rangle(\eta) &= \int_X \overline{(\phi f)(x)} g(x\eta) \;\dd\lambda_{r_H(\eta)}(x) \\
&= \int_X\int_G \overline{\phi(\gamma)} \,\overline{f(\gamma^{-1} x)} g(x\eta) \; \Delta^{1/2}(\gamma, \gamma^{-1} x) \; \dd \alpha^{r_X(x)}(\gamma)\;\dd\lambda_{r_H(\eta)}(x) \notag\\
&= \int_X\int_G \overline{f(\gamma^{-1} x)}\, \overline{\phi(\gamma)} g(x\eta) \;
\Delta^{1/2}(\gamma, \gamma^{-1} x) \; \dd
\alpha^{r_X(x)}(\gamma)\;\dd\lambda_{r_H(\eta)}(x) .\notag
\end{align}
Make a change of variables $(\gamma, \gamma^{-1} x) \mapsto
(\gamma^{-1}, x)$.  Then we use the fact that $\Delta$ is an almost
everywhere groupoid homomorphism (see Remark~\ref{rem:Delta-is-ae-homomorphism}). Due to Remark~\ref{rem:Delta-is-right-invariant}, we
know that $\Delta$ is $H$-\nb invariant. Taking into account these facts we see that
\begin{align*}
\langle \phi f, g\rangle(\eta) 
&= \int_X\int_G \overline{f(x)}\,\overline{\phi(\gamma^{-1})} g(\gamma^{-1}
  x\eta) \;\Delta^{1/2}(\gamma, \gamma^{-1} x\eta) \; \dd\alpha^{r(x)}(\gamma)\;\dd\lambda_{r_H(\eta)}(x) \\
&= \int_X\int_G \overline{f(x)} (\phi^*g)(x\eta)\;\dd\lambda_{r_H(\eta)}(x) \\
&= \langle f, \phi^* g\rangle(\eta) . \qedhere
\end{align*}
\end{proof}

It can be seen that the left and the right actions and the map $\inpro{\,}{\,}$ are continuous in the inductive limit topology.

Equations~\eqref{eq:bhel-1}, \eqref{eq:bhel-2} and~\eqref{eq:bhel-3} show that $C_c(X)$ is a
$C_c(G)$-\nb$C_c(H)$\nb-bimodule. Equations~\eqref{eq:bhel-4},
\eqref{eq:bhel-5} and~\eqref{eq:bhel-6} show that the map $\langle
\;,\,  \rangle: C_c(X) \times C_c(X) \to C_c(H)$ is
a $C_c(H)$\nb-conjugate bilinear map. And Equation~\eqref{eq:bhel-7}
says that $C_c(G)$ acts on $C_c(X)$ by $C_c(H)$\nb-adjointable operators.

\subsection{The right action---construction of the Hilbert module}
\label{sec:left-side-hil-module}

In this subsection, we describe how to construct a $\Cst(H,\beta)$\nb-Hilbert module $\Hils(X)$, where $X$ is a proper $H$\nb-space and $\lambda$ is an $H$\nb-invariant family of measures. Indeed, writing $\Hils(X,\lambda)$ is more precise than $\Hils(X)$. But we shall not come across any case which involves the same space with different families of measures. Hence we write $\Hils(X)$.

Note that to get $\Hils(X)$ from $C_c(X)$, we only need to prove that the bilinear map is positive. The other required properties of $\inpro{}{}$ are clear from Lemma~\ref{lemma:bhel}. 

The main result of this section, namely, Theorem~\ref{prop:proper-space-gives-hilmod} is a special case of the well-know theorem of Renault~\cite{Renault1985Representations-of-crossed-product-of-gpd-Cst-Alg}*{Corrolaire 5.2}. In~\cite{Renault1985Representations-of-crossed-product-of-gpd-Cst-Alg}, Renault uses a {\Star}category of measures free and proper \(H\)\nb-spaces to prove Corollaire~{5.2}. A remark following the corollaire says that the \emph{freeness} of the action is not a necessary hypothesis; the result holds even if this hypothesis is removed. Based on this remark, a similar category that uses measured proper \(H\)\nb-spaces is constructed in~\cite{Holkar-Renault2013Hypergpd-Cst-Alg}. In~\cite{Renault1985Representations-of-crossed-product-of-gpd-Cst-Alg} and~\cite{Holkar-Renault2013Hypergpd-Cst-Alg}, the main idea of the proof is that the representations of \((H,\beta)\) induce representations of this category. These induced representations are used to complete \(C_c(X)\) into a Hilbert \(\Cst(H,\beta)\)\nb-module. Thus Theorem~\ref{prop:proper-space-gives-hilmod} follows directly from the result concerning the completion of the {\Star}category. The proof below is written for the sake of completeness; the proof uses elementary techniques.
\begin{proposition}\label{prop:inner-product-is-positive}
  Let $(H,\beta)$ be a locally compact groupoid equipped with a Haar system, $X$ a proper left $H$-space and $\lambda$ an invariant family of measures on $X$. Then the bilinear map defined by Equation~\eqref{def:inner-product} is a $C_c(H)$-valued inner product on $C_c(X)$. 
\end{proposition}

 To prove the positivity of $\inpro{}{}$ we prove the following fact: for every (non-degenerate) representation $\tilde\pi\colon \Cst(H,\beta)\to \Bound(\Hils[K])$ on a Hilbert space $\Hils[K]$, $\tilde\pi(\inpro{f}{f})\in\Bound(\Hils[K])$ is positive. However, instead of working with the algebraic representations of $
\Cst(H,\beta)$ we work with representation of the groupoid $(H,\beta)$ and prove the positivity result. Since the integration and disintegration of representations establish an equivalence between the representations of $\Cst(H,\beta)$ and $(H,\beta)$, working with the representations of $(H,\beta)$ is enough.

Before proceeding to the proof we set up few notions and notation.
\begin{figure}[htb]
  \centering
\[
\begin{tikzcd}[scale=2]
X\times_{\base}H\arrow{r}{\lambda_1}[swap]{\pi_2} \arrow{d}[swap]{\tilde\beta_X}{\pi_1}
 & H \arrow{d}[swap]{r_H}{\beta} \\
 X \arrow{r}{\lambda}[swap]{s_X} 
& \base.
 \end{tikzcd}
\]
\caption{}
\label{fig:comm-square-for-inducing-measures-on-proper-quotient-3}
\end{figure}
We draw Figure~\ref{fig:comm-square-for-inducing-measures-on-proper-quotient-3} in which $\pi_i$ for $i=1,2$ are the projections on the $i^{\textnormal{th}}$ component, $\lambda_1$ and $\tilde\beta_X$ are families of measures as in Lemma~\ref{lemma:french-1.2}, that is, for $f\in C_c(X\times_{\base}H)$ and $x\in X$ the measure $\tilde\beta^x$ is
\begin{equation}
  \label{eq:tilde-beta-on-X*H}
 \int_{X\times_{\base}H} f\,\dd\tilde\beta^x=\int_{H} f(x,\eta)\,\dd\beta^{s_X(x)}(\eta). 
\end{equation}
And $\lambda_1$ is defined similarly. Clearly, $\beta\circ\lambda_1=\lambda\circ\tilde\beta_X$.
 
 Let $m$ be a transverse measure class on $H$. We take the quotient of each space in Figure~\ref{fig:comm-square-for-inducing-measures-on-proper-quotient-3} and the corresponding induced maps and families of measures. We do the following identifications: $(X\times_{\base}H)/H\homeo X, H/H\homeo \base, [\pi_2]=s_X$ and $[\pi_1]$ is the quotient map $X\to X/H$.
\begin{enumerate}[label=\roman*)]
\label{rem:equival-of-measure-in-transv-class}
\item The coherence of $m$ gives  
$m(\beta)\circ[\lambda_1]=m(\lambda)\circ[\tilde\beta_X]$, where $m(\beta)$ and $m(\lambda)$ are the measure classes induced by $m$ on $\base$ and $X/H$, respectively. 
\item A straight forward computation gives that $[\lambda_1]=\lambda$.
\item A computation as in (ii) above gives that
$[\tilde\beta_X]=\beta_X$, where $\beta_X$ is as in
Example~\ref{ex:def-of-beta-x}. 
\item Observations (i), (ii) and (iii) above together say that  
$m(\beta)\circ\lambda=m(\lambda)\circ\beta_X$.
\end{enumerate}
Hence if $\mu\in m(\beta)$ and $\nu\in m(\lambda)$, then
\begin{equation}
  \label{eq:equi-of-meausures}
\mu\circ\lambda\sim\nu\circ\beta_X.  
\end{equation}
 The symbol $\sim$ denotes the equivalence of measures.

Reader may refer to~\cite{mythesis}*{Section 1.3} for the computations in (ii) and (iii) above.

Now Proposition~\ref{prop:inner-product-is-positive} follows from Lemma~\ref{lemma:positivity-3} and Lemma~\ref{lemma:positivity-4} below. In the following discussion, we shall write $\inpro{f}{f}_{C_c(H)}$ instead of $\inpro{f}{f}$ for $f\in C_c(X)$.

Let $(m,\Hils,p_{\Hils},\pi)$ be a representation of $(H,\beta)$ where $m$ is a transverse measure class for $H$, $p_{\Hils}\colon\Hils\to\base$ is a measurable $H$\nb-Hilbert bundle which has separable fibres and $\pi$ is the action of $H$ on fibres of $\Hils$. 

By definition, the transverse measure class $m$ induces a measure
class $m(\lambda)$ on $X/H$. We fix $\mu\in m(\beta)$ and $\nu\in m(\lambda)$, that is, $\mu$ is a measure on~$H^0$ and~$\nu$ is a measure on~$X/H$.
Furthermore, let $\Lambda\colon C_c(X)\to C_c(H^0)$ and $B_X\colon C_c(X)\to C_c(X/H)$ be the integration operators for the families of measures $\{\lambda_u\}_{u\in H^0}$ and $\{\beta_X^{[x]}\}_{[x]\in X/H}$. Equation~\ref{eq:equi-of-meausures} shows that $\nu\circ\beta_X$ and $\mu\circ\lambda$ are equivalent measures on~$X$. Let $\tilde\beta_X$ be the family of measures
along the projection map $\pi_1\colon X\times_{\base}H\to X$, $\pi_1(x,h)= x$, as in Equation~\eqref{eq:tilde-beta-on-X*H}
\begin{lemma}
\label{lemma:positivity-1}
  The measure~$\nu\circ\beta_X$ on~$X$ is $H$-invariant.
\end{lemma}
\begin{proof}
  We must prove that the measure $\nu\circ\beta_X\circ\tilde\beta_X$ on $X\times_{\base}H$ defined by
  \[
  f\mapsto
  \int_{X/H} \int_{H^{s_X(x)}} \int_{H^{s_H(\eta)}} f(x\eta, h)
  \,{\dd}\beta^{s_H(\eta)}(h) \,{\dd}\beta^{s_X(x)}(\eta)
  \,{\dd}\nu[x]
  \]
  for $f\in C_c(X\times_{\base   } H)$ is invariant under the inversion map
  $(x,h)\mapsto (xh,h^{-1})$.  For this, we substitute $\eta^{-1}
  h$ for~$h$ and write
  \[
  \nu\circ\beta_X\circ\tilde\beta_X(f) = \int_{X/H} \int_{H^{s_X(x)}}
  \int_{H^{s_X(x)}} f(x\eta, \eta^{-1} h) \,{\dd}\beta^{s_X(x)}(h)
  \,{\dd}\beta^{s_X(x)}(\eta) \,{\dd}\nu[x];
  \]
  now replacing~$f$ by $f\circ \textup{inv}_{X\rtimes H}$ replaces $f(x\eta,
  \eta^{-1}h)$ by $f(x\eta \eta^{-1} h, (\eta^{-1} h)^{-1}) = f(x
  h, h^{-1} \eta)$.  The substitution that switches
  $h\leftrightarrow \eta$ shows that the integrals over $f$ and
  $f\circ \textup{inv}_{X\rtimes H}$ are the same.
\end{proof}

Since $\mu\circ\lambda$ is equivalent to $\nu\circ\beta_X$, this
measure on~$X$ must also be quasi-invariant.  We compute its
Radon-Nikodym derivative.  Let $f\in
C_c(X\times_{\base} H)$, then we get
\begin{align*}
  \mu\circ\lambda\circ\tilde\beta_X (f)
  &= \int_{H^0} \int_{X_u} \int_{H^u} f(x,h) \,{\dd}\beta^u(h)
  \,{\dd}\lambda_u(x) \,{\dd}\mu(u)
  \\&= \int_{H^0} \int_{H^u} \int_{X_u} f(x,h) \,{\dd}\lambda_u(x)
  \,{\dd}\beta^u(h) \,{\dd}\mu(u)
\end{align*}
by Fubini's Theorem.  When we replace $f$ by $f\circ
\textup{inv}_{X\rtimes H}$ and use the $H$-invariance of~$\lambda$, we get
\begin{align*}
  \mu\circ\lambda\circ\tilde\beta_X (f\circ\textup{inv}_{X
\rtimes H})
  &= \int_{H^0} \int_{H^u} \int_{X_u} f(x h,h^{-1})
  \,{\dd}\lambda_u(x)
  \,{\dd}\beta^u(h) \,{\dd}\mu(u)
  \\&= \int_{H^0} \int_{H^u} \int_{X_{s_H(h)}} f(x,h^{-1})
  \,{\dd}\lambda_{s_H(h)}(x)
  \,{\dd}\beta^u(h) \,{\dd}\mu(u)
  \\&= \int_{H^0} \int_{H_u} \int_{X_{r_H(h)}} f(x,h)
  \,{\dd}\lambda_{r_H(h)}(x)
  \,{\dd}\beta^{-1}_u(h) \,{\dd}\mu(u),
\end{align*}
where the last step uses the substitution $h\mapsto h^{-1}$. In terms of the integration operator $\Lambda_1\colon C_c(X\times_{\base} H)\to C_c(H)$ along~$\pi_2$, we may rewrite this as $\mu\circ\beta^{-1}(\Lambda_1(f))$, whereas
$\mu\circ\lambda\circ\tilde\beta_X (f) = \mu\circ\beta(\Lambda_1(f))$.
Thus the Radon-Nikodym derivative is
\[
\frac{{\dd}\,\textup{inv}_{X\rtimes H}^*(\mu\circ\lambda\circ\tilde\beta_X)}
{{\dd}(\mu\circ\lambda\circ\tilde\beta_X)} (x,h)
= \frac{{\dd}(\mu\circ\beta^{-1})} {{\dd}(\mu\circ\beta)}
(h).
\]
Now let
\[
M(x) = \frac{{\dd} (\mu\circ\lambda)}
{{\dd} (\nu\circ\beta)}.
\]
\begin{lemma}
\label{lemma:positivity-2}
  \label{lem:M_quasi-invariant}
  Let $x\in X$ and $h\in H$ satisfy $s_X(x)=r_H(h)$.  Then
  \[
  M(x h) = M(x) \frac{{\dd}(\mu\circ\beta^{-1})}
  {{\dd}(\mu\circ\beta)} (h).
  \]
\end{lemma}
\begin{proof}
  Let $g\in C_c(X\times_{\base}H)$ and let $f = \tilde\beta_X(g)$, that is, $f(x) = \int_{H^{s_X(x)}} f(x,h)\,{\dd}\beta^{s_X(x)}(h)$. By definition of the Radon-Nikodym derivative, we have
  \[
  \int_X f(x) \,{\dd}(\nu\circ\beta)(x)
  = \int_X f(x) M(x)^{-1} \,{\dd}(\mu\circ\lambda)(x).
  \]
  Thus
  \[
  \int_{X\times_{\base}H} g(x,h) \,{\dd}(\nu\circ\beta\circ\tilde\beta_X)(x,h)
  = \int_{X\times_{\base}H} g(x,h) M(x)^{-1} \,{\dd}(\mu\circ\lambda\circ\tilde\beta_X)(x,h).
  \]
  Since the measure $\nu\circ\beta$ is $H$-invariant, the left
  hand side is invariant under replacing $g$
  by~$g\circ\textup{inv}_{X\rtimes H}$. Hence so is the right-hand side, that
  is,
  \begin{align*}
    &\int_X g(x,h) M(x)^{-1}
    \,{\dd}(\mu\circ\lambda\circ\tilde\beta_X)(x,h)\\
    &= \int_X g(xh,h^{-1}) M(x)^{-1} 
    \,{\dd}(\mu\circ\lambda\circ\tilde\beta_X)(x,h)
    \\&= \int_X g(xh,h^{-1}) M(x h)^{-1} \frac{M(x h)}{M(x)}
    \,{\dd}(\mu\circ\lambda\circ\tilde\beta_X)(x,h).
  \end{align*}
  Letting $g'(x,h) = g(x,h) M(x)^{-1}$, we see that $M(x
  h)/M(x)$ has to be the Radon-Nikodym derivative
  \[
  \frac{M(x h)}{M(x)} =
  \frac{{\dd}(\textup{inv}_{X\rtimes H}^* \mu\circ\lambda\circ\tilde\beta_X)}
  {{\dd}(\mu\circ\lambda\circ\tilde\beta_X)}(x,h)
  = \frac{{\dd}(\mu\circ\beta^{-1})}
  {{\dd}(\mu\circ\beta)}(h).\qedhere
  \]
\end{proof}

Recall that $\mathcal{H}=(\mathcal{H}_x)_{x\in X}$ is a $\mu$-measurable
field of Hilbert spaces over~$\base$ equipped with a
representation~$\pi$ of~$H$.  The Hilbert space
$L^2(\base,\mu,\mathcal{H})$ consists of all $\mu$-measurable
sections $\xi\colon \base\to \mathcal{H}$ such that
\[
\int_{\base} \lVert\xi(u)\rVert_{\mathcal{H}_u}^2 \,{\dd}\mu(u) <\infty.
\]
This norm comes from an inner product on~$L^2(\base,\mu,\mathcal{H})$,
of course.

We pull back~$\mathcal{H}$ to a field~$s_X^*\mathcal{H}$ of Hilbert
spaces over~$X$ along $s_X\colon X\to \base$.  Then we take the
induced field of Hilbert spaces~$\mathcal{H}^X$ over~$X/H$ whose
$\mu\circ \lambda$-measurable sections are those sections~$\zeta$
of~$s^*\mathcal{H}$ that satisfy $\pi_h(\zeta(x h)) = \zeta(x)$
for all $x\in X$, $h\in H$ with $s_X(x)=r_H(h)$.  For~$\nu$ as
above, we define the Hilbert space $L^2(X/H,\nu,\mathcal{H}^X)$ to
consist of those sections~$\zeta$ of~$\mathcal{H}^X$ with
\[
\int_{X/H} \lVert \zeta(x)\rVert^2_{\mathcal{H}_{s_X(x)}}
\,{\dd}\nu[x]<\infty.
\]
The function $\lVert\zeta(x)\rVert^2_{\mathcal{H}_{s_X(x)}}$ is constant on
$H$-orbits and thus descends to~$X/H$ because $\pi_h(\zeta(x h))
= \zeta(x)$ and the operators~$\pi_h$ are unitary.  The norm
defining $L^2(X/H,\nu,\mathcal{H}^X)$ comes from an obvious inner
product.  Notice that an element of $L^2(X/H,\nu,\mathcal{H}^X)$ is
\emph{not} a function on~$X/H$.

Now we define the operator $|f\rangle\!\rangle$ from
$L^2(\base,\mu,\mathcal{H})$ to $L^2(X/H,\nu,\mathcal{H}^X)$ and its
adjoint~$\langle\!\langle f|$.  Let $\xi \in
L^2(\base,\mu,\mathcal{H})$ and $\zeta \in
L^2(X/H,\nu,\mathcal{H}^X)$.  Computations by Renault which are
discussed in~\cite{Holkar-Renault2013Hypergpd-Cst-Alg} or~\cite{mythesis}*{Section 3.3.1} lead to the following formulae for
$\langle\!\langle f|$ and $|f\rangle\!\rangle$:
\begin{align}
  \label{eq:llf}
  (|f\rangle\!\rangle\xi)([x]) &=
  \int_{H^{s_X(x)}} f(x \eta) \pi_\eta(\xi(s_H(\eta)))
  \sqrt{M(x \eta)} \,{\dd}\beta^{s_X(x)}(\eta),\\
  (\langle\!\langle f|\zeta)(u) &=
  \int_{X_u} \overline{f(x)}\zeta(x)
  \frac{1}{\sqrt{M(x)}} \,{\dd}\lambda_u(x).
\end{align}
Notice that
\[
\pi_h(|f\rangle\!\rangle\xi(x h))
= \int_{H^{s_X(x)}} f(x h \eta) \pi_{h \eta}(\xi(s_H(\eta)))
\sqrt{M(x h \eta)} \,{\dd}\beta^{s_X(x)}(\eta)
= |f\rangle\!\rangle\xi(x)
\]
by the substitution $h\eta\mapsto \eta$ because~$\beta$ is
left-invariant.  Thus $|f\rangle\!\rangle\xi$ is a section
of~$\mathcal{H}^X$. Let $C_c(X;\Hils)_0$ and $C_c(X;\Hils)$ have similar meanings as $C_c(X)_0$ and $C_c(X)$, respectively, as on page~\pageref{page:Cc}. If we pick $\xi,\zeta\in C_c(X;\Hils)$, then $|f\rangle\!\rangle\xi\in C_c(X/H)$ and $\langle\!\langle f|\zeta\in C_c(\base)$.  Hence our operators $|f\rangle\!\rangle$ and $\langle\!\langle f|$
are well-defined on dense subspaces.

\begin{lemma}
  \label{lemma:positivity-3}
 Let $\xi,\zeta\in C_c(X;\Hils)$. Then $\langle
  \zeta, |f\rangle\!\rangle\xi\rangle = \langle\, \langle\!\langle
  f|\zeta, \xi\rangle$, that is, $\langle\!\langle f|$ is formally
  adjoint to $|f\rangle\!\rangle$.
\end{lemma}
\begin{proof}
On the one hand,
  \begin{align*}
    \langle \zeta, |f\rangle\!\rangle\xi\rangle
    &= \int_{X/H} \langle \zeta(x),|f\rangle\!\rangle\xi(x)\rangle
    \,{\dd}\nu(x)
    \\&= \int_{X/H} \int_{H^{s_X(x)}} \langle \zeta(x), f(x\eta)
    \pi_\eta\xi(s_H(\eta))\rangle \sqrt{M(x\eta)}
    \,{\dd}\beta^{s_X(x)}(\eta) \,{\dd}\nu[x]
    \\&= \int_{X/H} \int_{H^{s_X(x)}} \langle \zeta(x\eta),
    \xi(s_H(\eta))\rangle f(x\eta) \sqrt{M(x\eta)}
    \,{\dd}\beta^{s_X(x)}(\eta) \,{\dd}\nu[x]
    \\&= \int_{X/H} \int_{H^{s_X(x)}} \langle \zeta(x),
    \xi(s_X(x))\rangle f(x) \sqrt{M(x)}
    \,{\dd}(\nu\circ\beta)(x),
  \end{align*}
  where we used $\pi_h\zeta(xh) = \zeta(x)$, the unitarity
  of~$\pi_h$, and the definition of the measure $\nu\circ\beta$
  on~$X$.  On the other hand,
  \begin{align*}
    \langle\, \langle\!\langle f|\zeta, \xi\rangle
    &= \int_{\base} \langle\, \langle\!\langle f|\zeta(u), \xi(u)\rangle
    \,{\dd}\mu(u)
    \\&= \int_{\base} \int_{X_u} \langle
    \overline{f(x)}\zeta(x)\sqrt{M(x)}^{-1}, \xi(u)\rangle
    \,{\dd}\lambda_u(x) \,{\dd}\mu(u)
    \\&= \int_X \langle \zeta(x), \xi(s_X(x))\rangle f(x) \sqrt{M(x)}^{-1}
    \,{\dd}(\mu\circ \lambda)(x)
    \\&= \int_X \langle \zeta(x), \xi(s_X(x))\rangle f(x) \sqrt{M(x)}^{-1}
    \frac{{\dd}(\mu\circ \lambda)}
    {{\dd}(\nu\circ\beta)}(x)
    \,{\dd}(\nu\circ\beta)(x).
  \end{align*}
  Now the definition of~$M$ shows that this is the same as the
  previous integral.
\end{proof}

The convolution algebra $C_c(H)$ acts on
$L^2(\base,\mu,\mathcal{H})$ by
\[
L(f)\xi(u) = \int_{H^u} f(\eta) \pi_\eta\xi(s_H(\eta))
\sqrt{\frac{{\dd}(\mu\circ\beta^{-1})}
{{\dd}(\mu\circ\beta)}}(\eta) \,{\dd}\beta^u(\eta).
\]
This is a $^*$-representation.
\begin{lemma}
\label{lemma:positivity-4}
  Let~$\xi\in C_c(X;\Hils)$.  Then $\langle\!\langle f|\circ
  |f\rangle\!\rangle(\xi) = L(\langle f,f\rangle_{C_c(H)})(\xi)$.
  Hence $\langle\!\langle f|\circ |f\rangle\!\rangle$ extends to a
  bounded operator with norm at most $\lVert\langle f,f\rangle
  \rVert_{C^*(H)}$.  It follows that $|f\rangle\!\rangle$ and
  $\langle\!\langle f|$ extend to bounded operators between the
  Hilbert spaces $L^2(\base,\mu,\mathcal{H})$ and
  $L^2(X/H,\nu,\mathcal{H}^X)$ which are adjoints of one another.
\end{lemma}
\begin{proof}
  We compute
  \begin{align*}
    \langle\!\langle f|\circ |f\rangle\!\rangle(\xi) (u)
    &= \int_{X_u} \overline{f(x)} |f\rangle\!\rangle(\xi) (x)
    \sqrt{M}^{-1}(x) \,{\dd}\lambda_u(x)
    \\&= \int_{X_u} \int_{H^u} \overline{f(x)} f(x\eta)
    \pi_\eta(\xi(s_H(\eta))) \sqrt{M}(x\eta) \sqrt{M}^{-1}(x)
    \,{\dd}\beta^u(\eta) \,{\dd}\lambda_u(x)
  \end{align*}
  Now we use Lemma~\ref{lem:M_quasi-invariant} to identify
  $M(x\eta)/M(x)$ with the function
  \[
  \delta(\eta) =
  \frac{{\dd}(\mu\circ\beta^{-1})}
  {{\dd}(\mu\circ\beta)} (\eta).
  \]
  Then we use Fubini's Theorem and continue the computation:
  \begin{align*}
    \langle\!\langle f|\circ |f\rangle\!\rangle(\xi) (u)
    &= \int_{H^u} \int_{X_u} \overline{f(x)} f(x\eta)
    \pi_\eta(\xi(s_H(\eta))) \sqrt{\delta(\eta)}
    \,{\dd}\lambda_u(x) \,{\dd}\beta^u(\eta)
    \\&= \int_{H^u} \langle f,f\rangle_{C_c(H)}(\eta)
    \pi_\eta(\xi(s_H(\eta))) \sqrt{\delta(\eta)}
    \,{\dd}\beta^u(\eta)
    = L(\langle f,f\rangle_{C_c(H)})(\xi).
  \end{align*}
  Since $L(\langle f,f\rangle_{C_c(H)})$ is bounded, it follows that
  $\langle\!\langle f|\circ |f\rangle\!\rangle$ extends to a bounded
  operator on $L^2(\base,\mu,\mathcal{H})$.  Let~$C>0$ be its norm.
  Then
  \[
  \lVert |f\rangle\!\rangle\xi \rVert^2 =
  \lvert\langle \xi, \langle\!\langle f|\circ
  |f\rangle\!\rangle\xi\rangle\rvert
  \le C \lVert\xi\rVert^2
  \]
  by Lemma~\ref{lemma:positivity-3} for all $\xi\in C_c(X,\Hils)$.  Hence~$|f\rangle\!\rangle$ extends to a bounded
  operator from $L^2(\base,\mu,\mathcal{H})$ to
  $L^2(X/H,\nu,\mathcal{H}^X)$.  A similar estimate shows that
  $\langle\!\langle f|$ extends to a bounded operator from
  $L^2(X/H,\nu,\mathcal{H}^X)$ to $L^2(\base,\mu,\mathcal{H})$.
\end{proof}

\begin{proof}[Proof of
  Proposition~\ref{prop:inner-product-is-positive}]
Follows from Lemma~\ref{lemma:positivity-4}.\hspace{5cm}\qedhere
\end{proof}

The last proposition shows that  $C_c(X)$ is a $\Cst(H,\beta)$\nb-pre-Hilbert module. Let
$\Hils(X)$ denote the $\Cst(H,\beta)$\nb-Hilbert module obtained by
completing $C_c(X)$.
\begin{theorem}[See~\cite{Renault1985Representations-of-crossed-product-of-gpd-Cst-Alg}*{Corrolaire 5.2}, or the discussion above]\label{prop:proper-space-gives-hilmod}
Let $(H,\beta)$ be a locally compact groupoid with a Haar system and let $X$ be a locally compact proper right $H$\nb-space carrying an $H$\nb-invariant continuous family of measures $\lambda$. Then using Formulae~\eqref{def:left-right-action} and~\eqref{def:inner-product} the right inner product $C_c(H)$\nb-module over the pre-\(\Cst\)\nb-algebra $C_c(X)$ can be completed to a $\Cst(H,\beta)$\nb-Hilbert module $\Hils(X)$.
\end{theorem}

\subsection{The left action and construction of the \texorpdfstring{$\Cst$}{C*}-correspondence}
\label{sec:left-act-main-thm}

Now we turn our attention to the left action. We wish to extend the action of
$C_c(G)$ on $C_c(X)$ to an action of $\Cst(G)$ on $\Hils(X)$. For a
groupoid equivalence the adjoining function vanishes (see Example~\ref{exa:ME-is-top-corr}), that is, it becomes the constant
function $1$, and the formulae for the left actions in
Definition~\ref{def:left-right-action} and~\cite{Muhly-Renault-Williams1987Gpd-equivalence} match. In this case, $C_c(G)$ acts on $C_c(X)$ by $\Cst(H,\beta)$\nb-adjointable operators. Our proof for the non-free case runs along the same lines
as in~\cite{Muhly-Renault-Williams1987Gpd-equivalence}.
\begin{lemma}\label{lemma:left-action-self-adj}
The action of\; $C_c(G)$ on $C_c(X)$ defined by
Equation~\ref{def:left-right-action} extends to a non-degenerate 
\Star{}homomorphism from~$\Cst(G,\alpha)$ to $\Bound_{\Cst(H)}(\Hils(X))$.
\end{lemma}
\begin{proof}
Equation~\ref{eq:bhel-7} in Lemma~\ref{lemma:bhel} shows that $C_c(G)$ acts by on $C_c(X)$ by adjointable operators. We need to prove that this representation is non-degenerate. For this purpose we show that there is continuous family of measures with full support $\tilde\alpha=\{\tilde\alpha_x\}_{x\in X}$ along the projection map $\pi_2\colon G\ltimes X\to X$. Then Lemma~\ref{lemma:french-1.1} gives that $\tilde A\colon C_c(G\times_{\base[G]}X)\to C_c(X)$ is a surjection. Since, due to the theorem of Stone-Weiersta{\ss} the set $\{g\otimes h: g\in C_c(G), h\in C_c(X)\}\subseteq C_c(G\times_{\base[G]}X)$ is dense, the claim of the current lemma will be proved.
 
  Let $f\in C_c(G\times_{\base[G]}X)$ is given. For $x\in X$ define the measure $\tilde\alpha_x$ on $\pi_2\inverse(x)=G^{r_X(x)}\times\{x\}$ by \[\int_{\pi_2\inverse(x)} f\,\dd\tilde\alpha^{x}=\int_{G^{r_X(x)}} f(\eta\inverse,x)\Delta^{1/2}(\eta\inverse,x)\,\dd\alpha^{r_X(x)}(\gamma).\] Since $\alpha^{r_X(x)}$ has full support and $\Delta$ is non-zero, $\tilde\alpha^{s_{G\times_{\base[G]}X}}$ has full support. Using an argument similar to Lemma~\ref{lemma:french-1.2} we may infer that $\tilde\alpha\defeq\{\tilde\alpha_x\}_{x\in X}$ is continuous.

Finally, we check that the action is bounded. Once we prove this, then the action extends to $\Cst(G,\alpha)$. Let $\epsilon$ be a state on $\Cst(H,\beta)$. Then $\epsilon(\inpro{}{})$ makes $\Hils(X)$ into a Hilbert space, say
$\Hils(X)_{\epsilon}$. Let $V_{\epsilon}\subseteq \Hils(X)_\epsilon$ be the dense subspace space generated by $\{\zeta f: \zeta \in C_c(G),\, f\in
C_c(X)\}$. Define a representation $L$ of $C_c(G)$ on $V_{\epsilon}$ by $L(\zeta)f = \zeta f$. 
\begin{enumerate}[label=\textit{\roman*})]
  \item The representation $L$ is a non-degenerate representation of
    $C_c(G)$ on $V_{\epsilon}$. Non\nb-degenerate means that the set
    $\{\zeta f: \zeta \in C_c(G),\, f\in C_c(X)\}$ is dense in
    $V_{\epsilon}$.
   \item The continuity of the operations in Lemma~\ref{lemma:bhel} in
     the inductive limit topology implies that $L$ is continuous: for
     $f,g\in C_c(X)$, $L_{f, g}(\zeta)= \langle f, L(\zeta)g
     \rangle$ is a continuous functional on $C_c(G)$ when $C_c(G)$ is given the
     inductive limit topology. 
     \item $L$ preserves the involution, that is, $\langle {\zeta}f,
       g\rangle =\langle f, \zeta^* g\rangle$. This is proved in
       Equation~\ref{eq:bhel-7} in Lemma~\eqref{lemma:bhel}.
\end{enumerate}
Proposition $4.2$
of~\cite{Renault1985Representations-of-crossed-product-of-gpd-Cst-Alg}
says that $L$ is a representation of $G$ on $V_{\epsilon}$. Hence $L$ is bounded with respect to the norm on $\Cst(G)$. Thus $\epsilon(\inpro{\zeta f}{\zeta f}=\epsilon(\inpro{L(\zeta) f}{L(\zeta) f} \leq |\!|\zeta|\!|_{\Cst(G)}\,
\epsilon(\langle \, f, f \rangle)$ for all $f \in C_c(X)$ and $\zeta
\in C_c(G)$. The state $\epsilon$ was arbitrary. Hence for all $f \in C_c(X)$ and $\zeta \in C_c(G)$ we get
\[
\langle \,\zeta f,\zeta f \rangle \leq |\!|\zeta|\!|_{\Cst(G)}\,  \langle \, f, f \rangle.
\]
\noindent This shows that the action of $C_c(G)$ on $C_c(X)$ is bounded in the
topology induced by the norm of the inner product $\langle
\,,\rangle$. Hence the action can be extended to $\Cst(G)$.
\end{proof}

Now we are ready to state the main theorem.
\begin{theorem}\label{thm:mcorr-gives-ccorr}
Let $(G, \alpha)$ and $(H, \beta)$ be locally compact groupoids with Haar systems. Then a topological correspondence $(X, \lambda)$ from $(G, \alpha)$ to $(H, \beta)$ produces a\; $\Cst$\nb-correspondence $\Hils(X)$ from $\Cst(G,\alpha)$ to $\Cst(H,\beta)$.
\end{theorem}
\begin{proof}
Follows by putting Proposition~\ref{prop:proper-space-gives-hilmod}
and Lemma~\ref{lemma:left-action-self-adj} together.
\end{proof}

 \section{Examples}
\label{cha:examples-theory}

\begin{example}
  \label{exa:cont-function-as-corr}
Let $X$ and $Y$ be locally compact, Hausdorff spaces, and let $f\colon  X\to Y$ be a continuous function. We view $X$ and $Y$ as groupoids with Haar systems consisting of Dirac measures on $X$ and $Y$, $\delta_X=\{\delta_x\}_{x\in X}$ and $\delta_Y=\{\delta_y\}_{y\in Y}$, respectively. We write $X'$ for the \emph{space} $X$. We use this notation to avoid confusing the space and the groupoid structures. 

The function $f$ is the momentum map for the trivial left action of $Y$ on $X'$, that is, for $(f(x),x)\in Y\times_{\textup{Id}_Y,Y,f} X'$, $f(x)\cdot x=x$. In fact, this is the only possible action of $Y$ on $X'$. In a similar way $X$ acts on itself trivially via the momentum map $\textup{Id}_X$. The family of Dirac measures $\delta_X$ mentioned above is an $X$\nb-invariant family of measures on $X'$. Both the actions are proper. If $h\in C_c(Y\times_{\textup{Id}_Y,Y,f} X')$, then
 
\begin{multline*}
\int_{X'}\int_Y h(y,x)\,\dd(\delta_Y)^{f(a)}(y)\,\dd(\delta_X)^a(x)=h(f(a),a)\\=\int_{X'}\int_Y h(y\inverse,yx)\,\dd(\delta_Y)^{f(a)}(y)\,\dd(\delta_X)^a(x).
\end{multline*}

Therefore $\delta_X$ is $Y$\nb-invariant.
Thus $(X',\delta_X)$ is a topological correspondence from $Y$ to~$X$ with the  constant function $1$ as the adjoining function.
The action of $C_c(X)$ on $C_c(X')$ as well as the $C_c(X)$\nb-valued inner product on $C_c(X')$ are the pointwise multiplication of two functions. For $h\in C_c(Y)$, $k\in C_c(X')$, $(h\cdot k)(x)=h(f(x))k(x)$.

The $\Cst$\nb-correspondence $\Hils(X')\colon C_0(Y)\to C_0(X)$ is the $\Cst$\nb-correspondence associated with the \Star{}homomorphism $f^*\colon C_0(Y)\to \Mult(C_0(X))$ produced by the Gelfand transform.
\end{example}

\begin{example}
  \label{exa:proper-map-as-corr}
Let $X$, $Y$, $X'$ and $f$ be as in Example~\ref{exa:cont-function-as-corr}. Let $\lambda=\{\lambda_y\}_{y\in Y}$ be a continuous family of measures along $f$. It follows from the discussion in Example~\ref{exa:cont-function-as-corr} that $X$ is a proper $X$-$Y$\nb-bispace. For $h\in C_c(X\times_{\textup{Id}_X,X,\textup{Id}_X}X')$,
\begin{multline*}
\int_{X'}\int_X h(x\inverse, xz)\,\dd(\delta_X)^x(z)\,\dd\lambda_y(x)\\=\int_{X'}\int_X h(x,z)\,\dd(\delta_X)^x(z)\,\dd\lambda_y(x)=\int_{X'} h(f(x),x)\,\dd\lambda_y(x).  
\end{multline*}
The first equality above is due to the triviality of the action and the second one follows from the definition of the measures $\delta_X\times\lambda_y$ as in Remark~\ref{rem:quasi-inv-of-lambda}. Thus $\lambda$ is  $X$\nb-invariant and the modular function is the constant function $1$. Hence $(X',\lambda)$ is a correspondence from $X$ to~$Y$.
\end{example}

\begin{example}
  \label{exm:topological-quiver-top-corr}
Let $X$ and $Y$ be as in Example~\ref{exa:cont-function-as-corr}. Let $f,b\colon  X\to Y$ be continuous maps and let $\lambda$ be a continuous family of measures along $f$. Make $X$ a $Y$\nb-$Y$\nb-bispace using actions similar to those in Example~\ref{exa:cont-function-as-corr}. For $f\colon  X\to Y$ use the family of measures 
$\lambda$ and the formulae in Example~\ref{exa:proper-map-as-corr} to define a right action of $C_c(Y)$ on $C_c(X)$. It is straightforward to check that $(X,\lambda)$ is a topological correspondence from $Y$ to itself. When the spaces are second countable, the quintuple $(Y, X,s,r,\lambda)$ is called a \emph{topological quiver} \cite{Muhly-Tomforde-2005-Topological-quivers}.

In general, assume that $X,Y$ and $Z$ are locally compact Hausdorff spaces, $f\colon X\to Y$ and $b\colon X\to Z$ are maps. Let $\lambda$ be a continuous family of measures along $f$. Then $(X,\lambda)$ is a topological correspondence from $Z$ to $Y$.
\end{example}

\begin{example}
  \label{exa:gp-homo-as-corr}
Let $G$ and $H$ be locally compact groups, $\phi\colon  H\to G$ a continuous group homomorphism, and $\alpha$ and $\beta$ the Haar measures on $G$ and $H$, respectively.  The right multiplication action is a proper action of $G$ on itself. The measure $\alpha\inverse$ is invariant under this action. Using $\phi$ define an action of $H$ on $G$ as $ \eta\gamma=\phi(\eta)\gamma$ for $(\eta,\gamma)\in H\times G$. We claim that $\alpha\inverse$ is $H$\nb-quasi-invariant for this $H$\nb-action. Let $\delta_G$ and $\delta_H$ be the modular functions of $G$ and $H$, respectively. The modular functions allow to switch between the left and right invariant Haar measures $\alpha$ and $\alpha\inverse$, and $\beta$ and $\beta\inverse$. The relations are $\alpha\inverse=\delta_G\inverse\alpha$ and $\beta\inverse=\delta_H\inverse\beta$. If $R_\gamma\colon G\to G$ is the right multiplication operator, then \[\int_G R_\gamma f\,\dd\alpha=\delta_G(\gamma)\inverse\int_G f\,\dd\alpha\] for $f\in C_c(G)$. A similar equality holds for $g\in C_c(H)$. Let $f \in C_c(H \times G)$,
\begin{align*}
  &\int_G\int_H f(\eta,\phi(\eta)\inverse\gamma)\, \frac{\delta_H(\eta)}{\delta_G(\phi(\eta))} \, \dd\beta(\eta)\,\dd\alpha\inverse(\gamma)\\
 &= \int_G\int_H f(\eta\inverse,\phi(\eta)\gamma)\,\frac{1}{\delta_G(\phi(\eta))}  \;\dd\beta(\eta)\,\dd\alpha\inverse(\gamma) \quad\textup{ (by sending }\eta\textup{ to }\eta\inverse \textup{ in } H\textup{)}\\
 &= \int_G\int_H f(\eta\inverse,\gamma) \; \dd\beta(\eta)\,\dd\alpha\inverse(\gamma)  \quad\textup{ (by removing } \phi(\eta)\inverse \textup{ in } G\textup{)}.
\end{align*}

If one compares the first term of the above computation with the equation in $(iv)$ of  Definition~\ref{def:correspondence}, and uses the fact that that the adjoining function is a groupoid homomorphism, then it can be seen that $\Delta(\eta, \eta\inverse \gamma) =\frac{\delta_H(\eta)}{\delta_G\circ\phi(\eta)}$. Hence $\Delta(\eta\inverse, \gamma)=\Delta(\eta, \eta\inverse \gamma)\inverse = \frac{\delta_G\circ\phi(\eta)}{\delta_H(\eta)}$. Thus a group homomorphism $\phi\colon  H\to G$ gives a topological correspondence $(G,\alpha\inverse)$ from $(H,\beta)$ to $(G,\alpha)$ and $\frac {\delta_G\circ\phi}{\delta_H}$ is the adjoining function.
\end{example}
\begin{example}
  \label{exa:proper-gp-homo-as-corr}
Let $G$, $H$, $\alpha$, $\beta$, $\delta_H$ and $\phi$ be as in Example~\ref{exa:gp-homo-as-corr}. Additionally, assume that $\phi\colon  H\to G$ is a proper function. For the time being, assume that the action of $H$ on $G$ given by $\gamma \eta\defeq \gamma\phi(\eta)$ for $(\gamma,\eta)\in G\times H$ is proper, which is a fact and we prove it towards the end of this example. With this action of $H$ and the left multiplication action of $G$ on itself, $G$ is a proper $G$-$H$\nb-bispace. 
$\alpha\inverse$ is an $H$\nb-invariant measure. The adjoining function of this action is the constant function $1$. To see this, let $f\in C_c(G\times G)$, then
\begin{align*}
 & \iint f(\gamma\inverse,\eta)\,\dd\alpha(\gamma)\dd\alpha\inverse(\eta)\\
&= \iint f(\gamma,\eta)\,\delta_G(\gamma)\inverse\dd\alpha(\gamma)\dd\alpha\inverse(\eta) &&\textnormal{(because } \alpha\inverse=\delta\inverse \alpha\textup{)}\\
&= \iint f(\gamma,\gamma\inverse\eta)\,\dd\alpha(\gamma)\dd\alpha\inverse(\eta) &&\textnormal{(because } L_\gamma\alpha\inverse=\delta(\gamma) \alpha\inverse\textup{)}.
\end{align*}
Now we prove that the action of $G$ on $H$ is proper, that is, the map $\Psi\colon  H\times G\to H\times H$ sending $(\gamma,\eta)\mapsto (\gamma,\gamma\phi(\eta))$ is proper. The maps
\begin{align*}
\textup{Id}_H\times\phi&\colon  H\times G\to H\times H \textup{ and }\\
m&\colon  (\eta,\eta')\mapsto (\eta,\eta\eta') \textup{ from } H\times H\to H\times H
\end{align*}

are proper, and $\Psi=m\circ (\textup{Id}_H\times\phi)$. Hence $\Psi$ is proper.
\end{example}

\begin{example}\label{exa:gp-to-point-correspondence}
Let $G$ be a locally compact group and $\alpha$ the Haar measure on $G$. Let $X$ be a locally compact proper left $G$-space. Let $\lambda$ be a \emph{strongly} $G$-quasi-invariant measure on $X$, that is, there is a continuous  function $\Delta: G\times X \rightarrow \R^{+} $ such that $\dd(g\lambda)(x) = \Delta(g, x) \dd\lambda(x)$ for every $g$ in $G$. In this setting, $(X,\lambda)$ is a correspondence from $(G,\alpha)$ to $(\textup{Pt}, \delta_{\textup{Pt}})$, with $\Delta$ as the adjoining function. The $\Cst$-\nb algebra for \textup{Pt} is $\C$, the Hilbert module $\Hils(X)$ is the Hilbert space $L^2(X,\lambda)$ and the action of $\Cst(G)$ on this Hilbert module is the representation of $\Cst(G)$ obtained from the representation of $G$ on $C_c(X)$.

    An example of this situation is: when $X$ is a homogeneous space for $G$, $X$ carries a $G$-strongly quasi-invariant measure. For details, see~\cite{Folland1995Harmonic-analysis-book}*{Section 2.6}.
\end{example}

\begin{example}[Macho Stadler and O'uchi's correspondences]\label{exm:Stadler-Ouchi-correspondence}
  In~\cite{Stadler-Ouchi1999Gpd-correspondences}, Macho Stadler and O'uchi present a notion of groupoid correspondences. We change the direction of correspondence in their definition to fit our construction and reproduce the definition here:
  \begin{definition}
    A correspondence from a locally compact, Hausdorff groupoid with Haar system $(G,\alpha)$ to a groupoid with Haar system $(H,\beta)$ is a $G$-$H$\nb-bispace $X$ such that:
    \begin{enumerate}[label=\roman*)]
    \item the action of $H$ is proper and the momentum map for the right action $s_X$ is open,
    \item the action of $G$ is proper,
    \item the actions of $G$ and $H$ commute,
    \item the right momentum map induces a bijection from $G\backslash X$ to $\base$.
    \end{enumerate}
 \end{definition}
Macho Stadler and O'uchi do not assume that the left momentum map is open. We do the same. Condition (iv) above is equivalent to saying that $G\7 X$ and $\base$ are homeomorphic.

Macho Stadler and O'uchi do not require a family of measures on the $G$-$H$-bispace $X$.  We show that a correspondence of Macho Stadler and O'uchi carries a canonical $H$\nb-invariant continuous family of measures $\lambda$ which is given by \[\int_{X_{u}} f\,\dd\lambda^u\defeq \int_G f(\gamma\inverse x)\,\dd\alpha^{r_X(x)}(\gamma)\;\quad \textup{for}\;\;f\in C_c(X),\]
where $u=s_X(x)$. Note that this family of measures is the family of measures $\alpha\inverse_X$ along the quotient map $X\to G\7X$ as in Example~\ref{ex:def-of-beta-x}. Condition (iv) in the above definition identifies $G\7X\homeo \base$ to give the desired result.

 Since $\lambda$ is invariant for the $G$\nb-actions, we get $\Delta=1$. Thus $(X,\lambda)$ is a \emph{topological correspondence} from $(G,\alpha)$ to $(H,\beta)$ in our sense.

Macho Stadler and O'uchi prove that such a correspondence from $(G,\alpha)$ to $(H,\beta)$ induces a $\Cst$-correspondence from $\Cred(G, \alpha)$ to $\Cred(H,\beta)$.
\end{example}
\begin{example}[Equivalence of groupoids]
\label{exa:ME-is-top-corr}
\begin{definition}[Equivalence of groupoids, a slight modification of
  Definition 2.1~\cite{Muhly-Renault-Williams1987Gpd-equivalence}]
\label{def:gpd-equivalence}
  Let $G$ and~$H$ be locally compact groupoids. A locally compact
   space $X$ is a $G$-$H$-equivalence if
  \begin{enumerate}[label=\roman*)]
  \item $X$ is a left free and proper $G$\nb-space;
  \item $X$ is a right free and proper $H$\nb-space;
  \item the momentum maps $r_X$ and $s_X$ are open;
  \item the actions of $G$ and $H$ commute, that is, $X$ is a $G$-$H$\nb-bispace;
  \item the left momentum map $r_X\colon  X\to\base[G]$ induces a bijection of
    $X/H$ onto $\base[G]$;
  \item the right momentum map $s_X\colon  X\to\base$ induces a bijection
    of $G\backslash X$ onto $\base[H]$.
  \end{enumerate}
\end{definition}
Equivalences of Hausdorff groupoids (\cite{Muhly-Renault-Williams1987Gpd-equivalence}) are a special case of Macho Stadler-O'uchi correspondences. Hence equivalences of groupoids are topological correspondences as well. Similarly, one can check that an equivalence of locally compact groupoids defined in~\cite{Renault1985Representations-of-crossed-product-of-gpd-Cst-Alg} is also a topological correspondence.  Furthermore, an equivalence of groupoids is an invertible correspondence.
\end{example}

\begin{example}[Generalised morphisms of Buneci and Stachura]
\label{exa:buneci-stachura}
Buneci and Stachura define \emph{generalised morphisms} in~\cite{Buneci-Stachura2005Morphisms-of-lc-gpd}. We modify this definition to fit our conventions and repeat it here:
\begin{definition}
A generalised morphism from $(G,\alpha)$ to $(H,\beta)$ is a left  action $\Theta$ of $G$ on the \emph{space} $H$ with $r_{GH}$ as the anchor map, the action commutes with the right multiplication action of $H$ on itself and there is a continuous positive function $\Delta_\Theta$ on $G\times_{s_G,\base,r_{GH}}H$ such that \[\iint f(\gamma,\gamma\inverse \eta)\,\Delta_\Theta(\gamma,\gamma\inverse \eta)\,\dd\alpha^{r_{GH}(\eta)}(\gamma)\,\dd\beta_u\inverse(\eta)=\iint f(\gamma\inverse, \eta)\,\dd\alpha^{r_{GH}(\eta)}(\gamma)\dd\beta_u\inverse(\eta)\]
for all $f\in C_c(G\times_{s_G,\base,r_{GH}}H)$ and $u\in\base$.
\end{definition}

If $\Theta$ is a generalised morphism from $(G,\alpha)$ to $(H,\beta)$ then $(H, \beta\inverse)$ is a topological correspondence from $(G,\alpha)$ to $(H,\beta)$, where $\beta\inverse$ is the family of measures
\[
\int f \,\dd(\beta^{-1})_u = \int f(\eta\inverse)\,\dd\beta^u(\eta)
\]
 for $f\in C_c(H)$ and $u\in\base$. It is obvious from the definition itself that $\Delta_\Theta$ is the adjoining function for this correspondence.

In~\cite{Buneci-Stachura2005Morphisms-of-lc-gpd},  Buneci and Stachura prove that a generalised morphism induces a \Star\nb{}homomorphism from $\Cst(G,\alpha)$ to $\Mult(\Cst(H,\beta))$. This is a $\Cst$\nb-correspondence from $\Cst(G,\alpha)$ to $\Cst(H,\beta)$ with the underlying Hilbert module $\Cst(H,\beta)$.
\end{example}

\begin{example}\label{ex:actiongpd-subgpd-correspondence}
  Let $X$ be a locally compact right $G$-space for a locally compact group $G$ and let $\lambda$ be the Haar measure on $G$. Let $H$ and $K$ be subgroups of $G$. Assume that $K$ is closed  and let $\alpha$ and $\beta$ be the Haar measures on $H$ and $K$, respectively. Then $X\rtimes H$ and $X\rtimes K$ are subgroupoids of $X\rtimes G$. Denote these three transformation groupoids by $\mathbf{H}$, $\mathbf{K}$ and $\mathbf{G}$, respectively. Then $\mathbf{G}$ is an $\mathbf{H}$\nb-$\mathbf{K}$-bispace for the left and the right multiplication actions, respectively. We bestow $\mathbf{H}$ and $\mathbf{K}$ with the Haar systems $\{\alpha^y\}_{y\in X}$ and $\{\beta^z\}_{z\in X}$, respectively, where $\alpha^y=\alpha$ and $\beta^z=\beta$ for each $y,z\in X$. If $\lambda\inverse_x =\lambda\inverse$ for all $x\in X$, then the family of measures $\{\lambda_x\inverse\}_{x\in X}$ on $\mathbf{G}$ is $\mathbf{K}$\nb-invariant. We show that this family is $\mathbf{H}$\nb-quasi-invariant with the adjoining function $\delta_G/\delta_H$.

For $((x,\gamma),(y, \kappa))\in \mathbf{G}\rtimes \mathbf{K}$ we have $y=x\gamma$ and the map $(x, \gamma, y, \kappa) \mapsto (x, \gamma, \kappa)$ gives an isomorphism between the groupoids $\mathbf{G}\rtimes \mathbf{K}$ and $X\rtimes (G\times K)$. Using this identification, it can be checked that the right action of $\mathbf{K}$ on $\mathbf{G}$ is proper, which is implied by the fact that $K\subset G$ is closed. Another quicker way to see this is to observe that 
$\mathbf{K}\subseteq\mathbf{G}$ is a closed subgroupoid.

Let $f\in C_c(\mathbf{H}\ltimes\mathbf{G})$, $u\in\base[\mathbf{K}]=\base[\mathbf{G}]\homeo X$. Let $b=(u\gamma\inverse,\gamma)\in s_\mathbf{G}\inverse(u)\subseteq \mathbf{G}$. If $a=(u\gamma\inverse,\eta)\in\mathbf{H}$, then $a\inverse=(u\gamma\inverse,\eta)\inverse=(u\gamma\inverse\eta,\eta\inverse)$ is composable with $b$ and $a\inverse b=(u\gamma\inverse,\eta)\inverse(u\gamma\inverse,\gamma)=(u\gamma\inverse\eta,\eta\inverse\gamma)$. Now a computation similar to that in Example~\ref{exa:gp-homo-as-corr} shows that
  \begin{align*}
&\int_{\mathbf{G}}\int_{\mathbf{H}} f(a,a\inverse b) \,\frac{\delta_H(\eta)}{\delta_G(\eta)}\,\dd\alpha^{r_{\mathbf{G}}(b)}(a) \,\dd\lambda\inverse_{u}(b)\\
 &=\int_{\mathbf{G}}\int_{\mathbf{H}} f((u\gamma\inverse,\eta),(u\gamma\inverse\eta, \eta\inverse\gamma)) \,\frac{\delta_H(\eta)}{\delta_G(\eta)}\,\dd\alpha^{r_{\mathbf{G}}(u\gamma\inverse,\gamma)}(u\gamma\inverse,\eta) \,\dd\lambda\inverse_{u}(u\gamma\inverse,\gamma)\\
 &= \int_{G}\int_{H} f((u\gamma\inverse,\eta),(u\gamma\inverse\eta, \eta\inverse\gamma)) \,\frac{\delta_H(\eta)}{\delta_G(\eta)}\,\dd\alpha(\eta) \,\dd\lambda\inverse(\gamma)\\
 &= \int_{G}\int_{H} f((u\gamma\inverse,\eta\inverse),(u\gamma\inverse\eta\inverse, \eta\gamma)) \frac{1}{\delta_G(\eta)} \,\dd\alpha(\eta) \,\dd\lambda\inverse(\gamma)\\&\quad\textnormal{(by changing } \eta\mapsto\eta\inverse \textup{)}.
 \end{align*}
Now we change $\gamma\mapsto \eta\inverse\gamma$. Then we use the relation $\dd\lambda\inverse(\eta\inverse\gamma)=\frac{\dd\lambda\inverse(\gamma)}{\delta_G(\eta)}$ to see that the previous term equals
 \begin{align*}
 & \int_{G}\int_{H} f((u\gamma\inverse\eta,\eta\inverse), (u\gamma\inverse, \gamma)) \,\dd\alpha(\eta) \,\dd\lambda\inverse(\gamma) \\
 &= \int_{\mathbf{G}}\int_{\mathbf{H}} f((u\gamma\inverse, \eta)\inverse, (u\gamma\inverse, \gamma)) \,\dd\alpha(u\gamma\inverse) \,\dd\alpha^{r_{\mathbf{G}}(u\gamma\inverse,\gamma)}(u\gamma\inverse,\eta) \,\dd\lambda\inverse_{u}(u\gamma\inverse,\gamma)\\
&= \int_{\mathbf{G}}\int_{\mathbf{H}} f(a\inverse, b) \,\frac{\delta_H(\eta)}{\delta_G(\eta)}\,\dd\alpha^{r_{\mathbf{G}}(b)}(a) \,\dd\lambda\inverse_{u}(b).
  \end{align*}
Thus $\{\lambda\inverse_x\}_{x\in X}$ is an $\mathbf{H}$\nb-quasi-invariant family of measures on $\mathbf{G}$ with $\delta_G/\delta_H$ as the adjoining function.
\end{example}
\medskip

Let $f'\colon X\to Y$ be a proper map between locally compact spaces. Let $B\subseteq Y$, $A\subseteq {f'}\inverse(B)$ and $f\colon A\to B$ be the map obtained by restricting the domain and the codomain of $f'$. Assume that $A\subseteq X$ is closed. Then we claim that $f$ is proper. Let $K\subseteq B$ be compact, then $K$ is compact in $Y$ also. Thus it is enough to consider the case when $B=Y$, in which case $f$ is the restriction of $f'$ to the closed subspace $A$. Since the inclusion map $i_A\colon A\hookrightarrow X$ is closed,~\cite{Bourbaki1966Topologoy-Part-1}*{Chapter I, \S 10.1, Proposition 2} shows that $i_A$ is proper. Hence $f=f'\circ i_A$ is proper.
\begin{example}
[The induction correspondence]
  \label{exm:general-induction}
Let $(G,\alpha)$ be a locally compact groupoid with a Haar system, $H$ a closed subgroupoid. Let $\beta$ be a Haar system for $H$. Note that $G_{\base}$ is a $G$-$H$-bispace where the left and right actions are multiplication from the left and right, respectively. Both actions are free. We claim that the actions of $G$ and $H$ are proper.

Let $\iota\colon \base[G]\to G$ be the inclusion map which is continuous. Then $\base=\iota\inverse(H)\subseteq \base[G]$ is closed.

  Let $\Psi\colon  G_{\base}\times_{\base} H\to G_{\base}\times G_{\base}$ be the map $\Psi(x,\eta)=(x, x\eta)$. Note that $\Psi$ is obtained from the proper map $(x,\eta)\mapsto (x, x\eta)$, $G\times_{s_G,\base[G],r_G}G\to G\times G$, by restricting the domain and codomain. If we prove that $G_{\base}\times_{\base}H\subseteq G\times_{s_G,\base[G],r_G}G$ is closed, then, from the discussion preceding this example, it will follow that $\Psi$ is proper.

Since $\base\subseteq \base[G]$ is closed, $(s_G\times r_G)\inverse(\base)=G\times_{s_G,\base,r_G} G\subseteq G\times_{s_G,\base[G],r_G} G$ is closed where $s_G\times r_G\colon G\times_{s_G,\base[G],r_G} G\to \base[G]$ is the map $(\gamma,\eta)\mapsto s_G(\gamma)=r_G(\eta)$. Now the projection on the second factor $\pi_2\colon G\times_{s_G,\base,r_G} G\to G$ is a continuous, due to which $G_{\base}\times_{\base} H=\pi_2\inverse(H)\subseteq G\times_{s_G,\base,r_G} G$ is closed.

To see that the left action is proper, first note that $G_{\base}=s_G\inverse(\base)\subseteq G$ is closed. Thus $\pi_2\inverse(G_{\base})=G\times_{s_G,\base[G],r_{G_{\base}}}G_{\base}\subseteq G\times_{s_G,\base[G],r_G}G$ is closed. And then arguing same as the right action shows that the left action is also proper.

Now it is not hard to see that $G\7 X\homeo \base$. By Example~\ref{exm:Stadler-Ouchi-correspondence}, $X$ produces a topological correspondence from $(G,\alpha)$ to $(H,\beta)$.

  Though both actions are free and proper, this correspondence need not be a groupoid equivalence as it might fail to satisfy Condition~{(v)} of Definition~\ref{def:gpd-equivalence}.
\end{example}

\paragraph{\textbf{Acknowledgement:}}
I thank Jean Renault and Ralf Meyer for many fruitful discussions. I thank Ralf Meyer specially; for proofreading the document carefully. I am thankful to the EuroIndia Project of Erasmus Mundus and CNPq, Brasil, that is, the Brazilian National Council for Scientific and Technological Development, for their supports.


\begin{bibdiv}
  \begin{biblist}
  \bib{Anantharaman-Renault2000Amenable-gpd}{book}{
   author={Anantharaman-Delaroche, C.},
   author={Renault, J.},
   title={Amenable groupoids},
   series={Monographies de L'Enseignement Math\'ematique [Monographs of
   L'Enseignement Math\'ematique]},
   volume={36},
   note={With a foreword by Georges Skandalis and Appendix B by E. Germain},
   publisher={L'Enseignement Math\'ematique, Geneva},
   date={2000},
   pages={196},
   isbn={2-940264-01-5},
   review={\MR{1799683 (2001m:22005)}},
}
\bib{Bourbaki1966Topologoy-Part-1}{book}{
   author={Bourbaki, Nicolas},
   title={Elements of mathematics. General topology. Part 1},
   publisher={Hermann, Paris; Addison-Wesley Publishing Co., Reading,
   Mass.-London-Don Mills, Ont.},
   date={1966},
   pages={vii+437},
   review={\MR{0205210 (34 \#5044a)}},
}
\bib{Buneci-Stachura2005Morphisms-of-lc-gpd}{article}{
 author={Buneci, M\u{a}d\u{a}lina Roxana},
  author={Stachura, Piotr},
  title={Morphisms of locally compact groupoids endowed with Haar systems},
  status={eprint},
  date={2005},
  note={arXiv {0511613}},
}
\bib{Folland1995Harmonic-analysis-book}{book}{
   author={Folland, Gerald B.},
   title={A course in abstract harmonic analysis},
   series={Studies in Advanced Mathematics},
   publisher={CRC Press, Boca Raton, FL},
   date={1995},
   pages={x+276},
   isbn={0-8493-8490-7},
   review={\MR{1397028 (98c:43001)}},
}

\bib{Holkar-Renault2013Hypergpd-Cst-Alg}{article}{
  author={Holkar, Rohit Dilip},
  author={Renault, Jean},
  title={Hypergroupoids and $C^*$\nobreakdash-algebras},
  language={English, with English and French summaries},
  journal={C. R. Math. Acad. Sci. Paris},
  volume={351},
  date={2013},
  number={23-24},
  pages={911--914},
  issn={1631-073X},
  review={\MR{3133603}{}},
  doi={10.1016/j.crma.2013.11.003},
}
\bib{Lance1995Hilbert-modules}{book}{
   author={Lance, E. C.},
   title={Hilbert $C^*$-modules},
   series={London Mathematical Society Lecture Note Series},
   volume={210},
   note={A toolkit for operator algebraists},
   publisher={Cambridge University Press, Cambridge},
   date={1995},
   pages={x+130},
   isbn={0-521-47910-X},
   review={\MR{1325694 (96k:46100)}},
   doi={10.1017/CBO9780511526206},
}
\bib{Muhly-Renault-Williams1987Gpd-equivalence}{article}{
   author={Muhly, Paul S.},
   author={Renault, Jean N.},
   author={Williams, Dana P.},
   title={Equivalence and isomorphism for groupoid $C^\ast$-algebras},
   journal={J. Operator Theory},
   volume={17},
   date={1987},
   number={1},
   pages={3--22},
   issn={0379-4024},
   review={\MR{873460 (88h:46123)}},
}
\bib{Muhly-Tomforde-2005-Topological-quivers}{article}{
   author={Muhly, Paul S.},
   author={Tomforde, Mark},
   title={Topological quivers},
   journal={Internat. J. Math.},
   volume={16},
   date={2005},
   number={7},
   pages={693--755},
   issn={0129-167X},
   review={\MR{2158956 (2006i:46099)}},
   doi={10.1142/S0129167X05003077},
}
\bib{mythesis}{thesis}{,
  title={Topological construction of $\textup{C}^*$-correspondences for groupoid ${C}^*$-algebras},
  author={Holkar, Rohit {D}ilip},
  year={2014},
  school={{G}eorg-{A}ugust-{U}niversit{\"a}t {G}\"ottingen}
}
\bib {PatersonA1999Gpd-InverseSemigps-Operator-Alg}{book}{
   author={Paterson, Alan L. T.},
   title={Groupoids, inverse semigroups, and their operator algebras},
   series={Progress in Mathematics},
   volume={170},
   publisher={Birkh\"auser Boston, Inc., Boston, MA},
   date={1999},
   pages={xvi+274},
   isbn={0-8176-4051-7},
   review={\MR{1724106 (2001a:22003)}},
   doi={10.1007/978-1-4612-1774-9},
}
\bib{Renault1980Gpd-Cst-Alg}{book}{
   author={Renault, Jean},
   title={A groupoid approach to $C^{\ast} $-algebras},
   series={Lecture Notes in Mathematics},
   volume={793},
   publisher={Springer, Berlin},
   date={1980},
   pages={ii+160},
   isbn={3-540-09977-8},
   review={\MR{584266 (82h:46075)}},
}
\bib{Renault1985Representations-of-crossed-product-of-gpd-Cst-Alg}{article}{
   author={Renault, Jean},
   title={Repr\'esentation des produits crois\'es d'alg\`ebres de groupo\"\i
   des},
   language={French},
   journal={J. Operator Theory},
   volume={18},
   date={1987},
   number={1},
   pages={67--97},
   issn={0379-4024},
   review={\MR{912813 (89g:46108)}},
}
\bib{Renault2014Induced-rep-and-Hpgpd}{article}{
   author={Renault, Jean},
   title={Induced representations and hypergroupoids},
   journal={SIGMA Symmetry Integrability Geom. Methods Appl.},
   volume={10},
   date={2014},
   pages={Paper 057, 18},
   issn={1815-0659},
   review={\MR{3226993}},
   doi={10.3842/SIGMA.2014.057},
}
\bib{Rieffel1974Induced-rep}{article}{
   author={Rieffel, Marc A.},
   title={Induced representations of $C^{\ast} $-algebras},
   journal={Advances in Math.},
   volume={13},
   date={1974},
   pages={176--257},
   issn={0001-8708},
   review={\MR{0353003 (50 \#5489)}},
}
\bib{Stadler-Ouchi1999Gpd-correspondences}{article}{
  author={Macho Stadler, Marta},
  author={O'uchi, Moto},
  title={Correspondence of groupoid $C^*$-algebras},
  journal={J. Operator Theory},
  volume={42},
  date={1999},
  number={1},
  pages={103--119},
  issn={0379-4024},
  eprint={www.mathjournals.org/jot/1999-042-001/1999-042-001-005.pdf},
}

\bib{Tu2004NonHausdorff-gpd-proper-actions-and-K}{article}{
   author={Tu, Jean-Louis},
   title={Non-Hausdorff groupoids, proper actions and $K$-theory},
   journal={Doc. Math.},
   volume={9},
   date={2004},
   pages={565--597 (electronic)},
   issn={1431-0635},
   review={\MR{2117427 (2005h:22004)}},
}
\bib{Westman1969Gpd-cohomology}{article}{
   author={Westman, Joel J.},
   title={Cohomology for ergodic groupoids},
   journal={Trans. Amer. Math. Soc.},
   volume={146},
   date={1969},
   pages={465--471},
   issn={0002-9947},
   review={\MR{0255771 (41 \#431)}},
}
  \end{biblist}
\end{bibdiv}
\end{document}